\documentclass[smallextended,referee,envcountsect,]{svjour3}
\smartqed
\usepackage{graphicx} 
\usepackage[top=1.2in, bottom=1.2in, left=1.25 in, right=1.25 in]{geometry}
\usepackage{amsmath}
\usepackage{amssymb}
\usepackage{booktabs} 
\usepackage[figuresright]{rotating}
\usepackage{algorithm}
\usepackage{algorithmic}
\journalname{JOTA}

\begin{document}

\title{A Vectorized Positive Semidefinite Penalty Method for Unconstrained Binary Quadratic Programming}

\subtitle{}

\author{Xinyue Huo and Ran Gu}

\institute{Xinyue Huo \at
             Nankai University \\
              School of Statistics and Data Science, Nankai University, Tianjin 300071, China\\
              xy\underline{~}h@mail.nankai.edu.cn
           \and
              Ran Gu \at
              Nankai University \\
              School of Statistics and Data Science, Nankai University, Tianjin 300071, China\\
              rgu@nankai.edu.cn
}

\date{Received: date / Accepted: date}

\maketitle

\begin{abstract}
The unconstrained binary quadratic programming (UBQP) problem is a class of problems of significant importance in many practical applications, such as in combinatorial optimization, circuit design, and other fields. The positive semidefinite penalty (PSDP) method originated from research on semidefinite relaxation, where the introduction of an exact penalty function improves the efficiency and accuracy of problem solving. In this paper, we propose a vectorized PSDP method for solving the UBQP problem, which optimizes computational efficiency by vectorizing matrix variables within a PSDP framework. Algorithmic enhancements in penalty updating and initialization are implemented, along with the introduction of two algorithms that integrate the proximal point algorithm and the projection alternating BB method for subproblem resolution. Properties of the penalty function and algorithm convergence are analyzed. Numerical experiments show the superior performance of the method in providing high-quality solutions and satisfactory solution times compared to the semidefinite relaxation method and other established methods.
\end{abstract}


\section{Introduction}

This study focuses on unconstrained binary quadratic programming (UBQP), a specialized form of integer programming. In this type of programming, the objective function takes a quadratic form and the variables are constrained to binary values of 0 and 1. This class of problems is considered vital and is often formulated as follows:
\begin{equation}
\min_{x \in \{0,1\}^n} x^T Q x + b^T x,
\end{equation}
where $x = [x_1, x_2, \ldots, x_n]^T \in  R^n$, $Q\in R^{n \times n}$, and $c\in R^n$. This problem can also be viewed as a special type of quadratically constrained quadratic programming (QCQP) problem, where the 0-1 constraint is written as $x_i^2 - x_i = 0$.

This programming has a wide range of research significance in applications, such as combinatorial optimization~\cite{Boros1991}, financial investments~\cite{Cesarone_2014}, signal processing~\cite{1363998}, production scheduling~\cite{TEIXEIRA2010425} and other fields. In general, this problem is NP-hard~\cite{PARDALOS1992119} due to the binary constraint. 

\subsection{Related works}
Due to the potential of the unconstrained quadratic programming problem, many authors have studied the optimality conditions and sought high-quality solution methods for the UBQP problem in recent years. Beck and Teboulle~\cite{doi:10.1137/S1052623498336930} derived the sufficient and necessary global optimality conditions. Based on the work of Beck and Teboulle, Xia~\cite{Xia2009NewOC} obtained tighter sufficient optimality conditions and explored the relationship between the optimal solution of the problem and that of its continuous relaxation. These studies have provided us with a deeper understanding of UBQP problems, leading to the development of efficient solution strategies.

One of the main categories of methods for solving UBQP is exact solvers. Typically, this is achieved by a tree-search approach that incorporates a generalized branch-and-bound technique. In the 1980s, both Gulati et al.~\cite{GULATI1984121} and Barahona et al.~\cite{1989Experiments} described a branch-and-bound algorithm for UBQP. More recently, Krislock et al.~\cite{Krislock2017BiqCrunchAS} present BiqBin, an exact solver for binary quadratic problems. The approach is based on an exact penalty method that efficiently transforms the original problem into an instance of Max-Cut, and then solves the Max-Cut problem using a branch-and-bound algorithm. The numerical results show that BiqBin is a highly competitive solver.

Heuristic algorithms are also often used to solve UBQP problems exactly. The most common are tabu search~\cite{inproceedings}, quantum annealing~\cite{PMID:35140264}. For example, Wang et al.~\cite{10.1007/978-3-642-29828-8_26} introduce a backbone-guided tabu search algorithm that involves a fundamental tabu search process alternating with a variable fixing/freeing phase, which is determined by the identification of strongly determined variables. Due to the NP-hardness, the optimal solution usually cannot be found in polynomial time. Therefore, it is necessary to find a fast method that is not strictly exact. Based on this, many scholars have proposed a number of inexact methods.

Among the inexact methods, there are mainly methods based on relaxation and penalty function methods. There are some penalty function methods based on vector variables. Liu et al.~\cite{Liu2018ACA} and Nayak et al.~\cite{RePEc:spr:jcomop:v:39:y:2020:i:3:d:10.1007_s10878-019-00517-8} have adopted a continuous approach, converting binary constraints into continuous ones by using Fischer-Burmeister nonlinear complementarity problem functions, then it can be further carried on the smoothing processing by aggregate function. However, this method is relatively sensitive to the initial values, and the parameter updates are complex and may not serve as exact penalty. Furthermore, the exact penalty method is widely used in many methods, transforming the binary constraints into a set of inequality constraints. Yuan et al.~\cite{yuan2017exact} introduced an alternative approach by replacing the binary constraints with a Mathematical Programming with Equilibrium Constraints (MPEC). This transformation converts the UBQP problem into a biconvex problem. While this algorithm demonstrates high computational efficiency, the solution accuracy still needs to be improved for certain data sets. Luo et al.~\cite{luo2019enhancing} introduced a special penalty method called conditionally
quasi-convex relaxation based on
semidefinite relaxation and they showed that the constructed
penalty function is exact. However, the subproblem of this algorithm still needs to solve an SDP subproblem, which is time-consuming.

Matrix-based relaxation methods are relatively common approaches, with typical examples including the reconstruction-linearization technique~\cite{Sherali2007RLTAU} and positive semidefinite relaxation (SDR)~\cite{fujie1997semidefinite}. SDR is a classical method for solving QCQP, which transforms quadratic terms of vector variables into matrix variables, drops the rank-one constraint, relaxes the problem to a semidefinite programming (SDP) problem, and utilizes SDP solvers for the solution. In many cases, SDR can obtain the exact solution to QCQP or prove to achieve a certain approximation ratio~\cite{2023Exact,5447068}. It also has wide applications in practice. 

When the matrix solution of SDR is not rank one, the traditional strategies are to use projection or random sampling~\cite{5447068}. Recently, other methods have also been studied to further obtain high-quality rank one solutions from SDR solutions~\cite{10.1093/imanum/draa031,article}. Gu et al. provided a broader perspective to SDR research by introducing a positive semidefinite-based penalty (PSDP) formulation~\cite{10.1093/imanum/draa031}. When the penalty factor is zero, it corresponds exactly to the semidefinite relaxation of QCQP problems. Moreover, if the penalty factor is sufficiently large, the penalty function problem is equivalent to QCQP. Therefore, by gradually increasing the penalty factor, the matrix solution of SDR can reach rank one. In addition, the authors used a proximal point algorithm (PPA) to solve the penalty function subproblems, which demonstrated remarkable performance when applied to UBQP, with a significant proportion yielding exact solutions for UBQP.

While the solution quality provided by PSDP is high, its significant computational cost as an SDR-based approach cannot be overlooked compared to vector methods. Therefore, a crucial research question is how to maintain the solution quality of PSDP while significantly reducing its computational cost, ideally to a form comparable to the computational efficiency of vector methods.

\subsection{Our contribution}
In the PSDP framework, we select a specified initial value for the penalty factor and find that the penalty function problem can depend only on the diagonal elements of the matrix. Following the penalty factor update rule of PSDP, the generated sequence of PSDP penalty problems still ensures that it is only related to the matrix diagonal elements, thus causing the matrix problem in PSDP to degrade to a vector problem, achieving the vectorization of PSDP. Furthermore, in the vectorized PSDP algorithm we propose, the algorithm for solving subproblem and penalty update rule maintain the same form as PSDP, while incorporating some improvements in the details. Therefore, the vectorized PSDP not only maintains the quality of solutions in PSDP, but also keeps the computational cost at the level of vector methods.

To further accelerate the convergence and improve the solution accuracy, we apply an algorithm accelerated by the projection alternating Barzilai-Borwein (BB) method~\cite{2005Projected}. This method uses the projection step to effectively handle the 0-1 constraints, while the alternating BB method provides a robust and efficient framework for updating the search direction and step size. Additionally, we introduce a criterion to update the penalty parameter during the iterations. This ensures that the active set is appropriately updated throughout the iteration process, thereby ensuring the convergence of the algorithm. We tested our algorithm on one randomly generated dataset and five benchmark datasets, comparing its performance with four existing methods. Experimental results show that our algorithm performs well in terms of both time efficiency and accuracy. 

\subsection{Organization}
In section~\ref{sec2}, we will present our vectorized exact penalty function and its transformation process. In section~\ref{sec3}, we will present our subproblem, two methods for solving the subproblem, and an approach for updating the penalty factor. Section~\ref{sec4} will provide the theoretical foundation of our algorithm, including the exact penalty function theory and the convergence of the subproblem algorithm. Finally, in section~\ref{sec5}, we will demonstrate the comparison of our algorithm with several other methods across different datasets through numerical experiments.
\section{Vectorization of PSDP}\label{sec2}
In this paper, we focus on the following UBQP problem:
\begin{equation}\label{bqp}
\begin{array}{cl}
\min & x^{T}Qx+b^Tx \\
\text { s.t. } &x \in \{0,1\}^n ,
\end{array}
\end{equation}
where $Q\in \mathbb{S}^{n \times n} $ is a real symmetric matrix, $b\in \mathbb{R}^{n}$  is a vector. We can reformulate \eqref{bqp} as a problem with alternative equivalent equation constraints:
\begin{equation}\label{o1}
\begin{array}{cl}
\min &   x^{T}Qx+b^Tx \\
\text { s.t. } &  x_{i}^{2}- x_i=0, \quad i=1, \ldots, n .
\end{array}
\end{equation}
Therefore, UBQP can be viewed as a QCQP problem subject to specific constraints.

\subsection{PSDP}
For a general QCQP as in \eqref{qcqp}, PSDP~\cite{10.1093/imanum/draa031} adds a matrix variable $Z\succeq 0$ and a penalty term $P\cdot Z$ with $P\succeq 0$ to the objective, proposing the following penalty problem \eqref{P1}.
\begin{equation}\label{qcqp}
\begin{aligned}
\min _{x \in \mathbb{R}^{n}} & \left(\begin{array}{cc}
1 & x^{\mathrm{T}} \\
x & x x^{\mathrm{T}}
\end{array}\right) \cdot\left(\begin{array}{cc}
0 & g_{0}^{\mathrm{T}} \\
g_{0} & Q_{0}
\end{array}\right) \\
\text { s.t. } & \left(\begin{array}{cc}
1 & x^{\mathrm{T}} \\
x & x x^{\mathrm{T}}
\end{array}\right) \cdot\left(\begin{array}{cc}
c_{i} & g_{i}^{\mathrm{T}} \\
g_{i} & Q_{i}
\end{array}\right)=0, \quad i=1, \ldots, m_{e}, \\
& \left(\begin{array}{cc}
1 & x^{\mathrm{T}} \\
x & x x^{\mathrm{T}}
\end{array}\right) \cdot\left(\begin{array}{cc}
c_{i} & g_{i}^{\mathrm{T}} \\
g_{i} & Q_{i}
\end{array}\right) \geq 0, \quad i=m_{e}+1, \ldots, m .
\end{aligned} 
\end{equation}

\begin{equation}\label{P1}
\begin{aligned}
\min _{x \in \mathbb{R}^{n}, Z \in \mathbb{S}^{n}} & \left(\begin{array}{lc}
1 & x^{\mathrm{T}} \\
x & x x^{\mathrm{T}}+Z
\end{array}\right) \cdot\left(\begin{array}{cc}
0 & g_{0}^{\mathrm{T}} \\
g_{0} & Q_{0}
\end{array}\right)+P \cdot Z \\
\text { s.t. } & \left(\begin{array}{cc}
1 & x^{\mathrm{T}} \\
x & x x^{\mathrm{T}}+Z
\end{array}\right) \cdot\left(\begin{array}{ll}
c_{i} & g_{i}^{\mathrm{T}} \\
g_{i} & Q_{i}
\end{array}\right)=0, \quad i=1, \ldots, m_{e}, \\
& \left(\begin{array}{cc}
1 & x^{\mathrm{T}} \\
x & x x^{\mathrm{T}}+Z
\end{array}\right) \cdot\left(\begin{array}{ll}
c_{i} & g_{i}^{\mathrm{T}} \\
g_{i} & Q_{i}
\end{array}\right) \geq 0, \quad i=m_{e}+1, \ldots, m, \\
& Z \succeq 0.
\end{aligned}
\end{equation}
When $P$ is 0, \eqref{P1} is equivalent to a standard SDR. 

In the UBQP problem, with $Q_0=Q$, $Q_i=e_i e_i^T$ ($e_i$ is a vector with 1 in the $i $-th position and 0 elsewhere), $g_0=0.5b_i$, $g_i=-0.5e_i$ and $c_i=0$ as in \eqref{qcqp}, the penalty problem \eqref{P1} corresponds to the following problem.
\begin{equation}
\begin{aligned}\label{1}
\min _{x \in \mathbb{R}^{n}, Z \in \mathbb{S}^{n}} & x^TQx+b^Tx+(Q+P) \cdot Z \\
\text { s.t. } & x_i^2-x_i+z_i=0,i=1,2,\dots,n,\\
& Z \succeq 0,
\end{aligned}
\end{equation}
where $z_i$ is the $i$-th diagonal element of $Z$.

PSDP employed a PPA method to solve problem \eqref{1}, with each PPA iteration involving the solution of an SDP problem, necessitating the repeated solving of SDPs. PSDP begins with $P=0$, which is equivalent to SDR, and then updates the penalty factor $P$ by $P = P+\alpha Z$.

\subsection{Vectorized PSDP}
Solving SDP is the major computational obstacle in PSDP, so we consider what kind of problem structure can make problem \eqref{1} easy to solve. 

We note that when $Q+P$ is a diagonal matrix, the solution of $Z$ depends only on its diagonal elements. This is because the first constraint involves only the diagonal elements of $Z$, and when $Q+P$ is a diagonal matrix, the objective function also involves only the diagonal elements of $Z$. For the remaining positive semidefinite constraints on $Z$, as long as its diagonal elements are non-negative and the non-diagonal elements are set to zero, it naturally satisfies the constraints. Therefore, we only need to make $Q+P$ a diagonal matrix so that there exists a diagonal matrix solution for $Z$. In addition, since $Z$ is a diagonal matrix, $Q+P$ actually only updates the diagonal elements when updating $P=P+\alpha Z$. Therefore, the new penalty problem still has a diagonal matrix solution for $Z$.

The construction of $P$ is straightforward that we simply need to set the off-diagonal elements of $P$ to be the negation of the off-diagonal elements of $Q$. Supposing that we choose an appropriate $P$ such that $Q + P$ also becomes a diagonal matrix $P^{\star}$,  problem \eqref{1} is thus transformed into the following form. 
\begin{equation}\label{p1}
\begin{array}{cc}
\min _{x \in \mathbb{R}^{n}} & x^T Q x+b^T x +\sum_{i=1}^{n}p_i(x_i-x_i^2)\\
\text { s.t. } & 0\leq x_i\leq 1, \quad i=1, \ldots, n, \\
\end{array} 
\end{equation}
where $p_i$ is the $i$-th diagonal element of $P^{\star}$, the inequality $x_i^2-x_i\leq 0$ has been replaced with $0\leq x_i\leq1$, and $Z$ has been eliminated. 

Billionnet and Elloumi~\cite{Billionnet2007UsingAM} have discussed the form of problem \eqref{p1}, exploring how to construct an appropriate $p$ to preserve the convexity of problem \eqref{p1} while also preserving equivalence with problem \eqref{bqp} as much as possible. In contrast, we continue to follow the PSDP approach of updating $p$, allowing the quadratic problem \eqref{p1} to be non-convex, and solving it using iterative algorithms.

\section{Algorithms}\label{sec3}
 Since problem \eqref{p1} can lead to a non-convex quadratic problem, finding the global optimal solution is not easy. We present two algorithms. First, we propose a PPA based on an improvement of the PPA in PSDP. Then, we introduce an accelerated method, namely the projection alternating BB method. Other algorithmic details, including the subproblem termination criterion, penalty factor updates, and initialization, are also included in this section.
 
\subsection{PPA}
In PSDP, when applying the PPA to solve \eqref{1}, the $k$-th iteration is carried out in the following form.
\begin{equation}\label{2}
\begin{array}{rl}
\left(d^{k}, Z^{k+1}\right)=\underset{d \in \mathbb{R}^{n}, Z \in \mathbb{S}^{n}}{\arg \min } & H\left(x^{k}+d, Z, P\right)+d^{\mathrm{T}} P d \\
\text { s.t. } & \left(x^{k}+d, Z\right) \in G_1,
\end{array}
\end{equation}
where $G_1=\{
(x,Z)|x_i^2-x_i+z_i=0,i=1,2,\dots,n, Z \succeq 0\}$ and $H \left(x^{k}+d, Z, P\right)=x^TQx+b^Tx+(Q+P) \cdot Z $. Its proximal term is $d^TPd$. For the vectorized problem \eqref{p1}, we observe the implementation of this PPA as follows:
\begin{equation}\label{dp1}
\begin{array}{rl}
d^{k}=\underset{d \in \mathbb{R}^{n}}{\arg \min  } & h\left(x^{k}+d,p\right)+d^T Pd \\
\text { s.t. }& x^{k}+d\in F_1 ,\\
\end{array}
\end{equation}
where $F_1=\{x|0\leq x_i\leq 1\}$, $h\left(x,p\right)=x^{T} (Q-P^{\star}) x+(b+p)^{T}x$, and $p$ is the vector composed of diagonal elements of $P^{\star}$.

By utilizing the nature of $Q+P=P^\star$, it can be observed that the objective function of \eqref{dp1} is a linear function of $d$. It is known that minimizing a linear function over box constraints leads to all components of the optimal solution being taken on the boundaries. This causes $x$ to reach the critical point of \eqref{p1} in one iteration. However, it is typically a low-quality local solution, since problem \eqref{p1} is non-convex. Therefore, we have to modify the proximal term to better solve \eqref{p1} using PPA.

Building on \eqref{dp1}, we introduce an additional quadratic term $d^T H_p d$ to the existing proximal term, where $H_p \succ 0$ to ensure strict convexity of the new problem and avoid linearity. Furthermore, we need to ensure that the new proximal term $(H_p + P)$ is positive semidefinite since we do not strictly require $P \succeq 0$ here. Lastly, it is necessary for $H_p$ to be a diagonal matrix to facilitate the computation of the subproblems. Therefore, we define $H_p$ in the following form.
$$
H_p=\begin{bmatrix}
  \max(\sum|Q_{1i}|-p_1,\epsilon)&0& \cdots & 0\\
  0& \max(\sum|Q_{2i}|-p_2,\epsilon) & \cdots & 0\\
  \vdots & \vdots & \ddots &\vdots \\
 0 & 0& \cdots & \max(\sum|Q_{ni}|-p_n,\epsilon)
\end{bmatrix},
$$
where $\epsilon >0 $ is an arbitrarily small number. Thus, $H_p\succ 0$ and
$$
H_p+P = H_p-Q+P^{\star}
\succeq \begin{bmatrix}
  \sum_{i\neq 1}|Q_{1i}|& -Q_{12}& \cdots & -Q_{1n}\\
  -Q_{21}& \sum_{i\neq2}|Q_{2i}|& \cdots & -Q_{2n}\\
  \vdots & \vdots & \ddots &\vdots \\
 -Q_{n1} &-Q_{n2} & \dots &\sum_{i\neq n}|Q_{ni}|
\end{bmatrix}
\succeq 0,
$$
where the final step utilizes the fact that a diagonally dominant matrix is positive semidefinite.

Setting the proximal term to be $(H_p+P=H_p-Q+P^\star)$, we obtain the $k$-th iteration as follows:
\begin{equation}\label{dp}
\begin{array}{rl}
d^{k}=\underset{d \in \mathbb{R}^{n}}{\arg \min } & h\left(x^{k}+d,p\right)+d^T (H_p-Q+P^{\star})d \\
\text { s.t. }& x^{k}+d\in F_1 \\
\end{array}
\end{equation}
In fact, 
$$h\left(x^{k}+d,p\right)+d^T (H_P-Q+P^\star)d=h\left(x^{k},p\right)+\nabla h\left(x^{k},p\right)^Td+d^TH_pd.$$ 
Thus, \eqref{dp} is separable in variables, equivalent to solving n one-dimensional problems. Without considering box constraints, the optimal solution is as follows:
\begin{equation}
    \tilde d^k_j=\frac{-\nabla h\left(x^{k},p\right)^j}{ 2h_p^j},
\end{equation} 
where $\tilde d^k_j$ is $j$-th element of $\tilde d^k$, $\nabla h\left(x^{k},p\right)^j$ is the $j$-th element of $\nabla h\left(x^{k},p\right)$ and $h_p^j$ is the $j$-th diagonal element of $H_p$. Then $d^k$ can be obtain by projecting $\tilde d^k$ onto box set $(F_1-x^k)$, that is,
\begin{equation}\label{d2}
    d^k = \mathcal P_{(F_1-x^k)}(\tilde d^k).
\end{equation}
Therefore, the next iterative point $x^{k+1}=x^{k}+d^{k}$ is given by 
\begin{equation}\label{x1}
    x^{k+1}=\mathcal P(x^{k}+\tilde d^{k}),
\end{equation}
where $\mathcal P$ is a projection operator defined by 
$$
\mathcal P(x)=\left \{   \begin{array}{l}
0,x<0,\\
x,0\leq x \leq 1,\\
1,x>1.
\end{array}
\right.
$$

We present the algorithmic framework of this PPA in Algorithm~\ref{algo1}, which is used to solve the penalty problem \eqref{p1} with a given penalty factor $p$.

\begin{algorithm}
\caption{PPA for subproblem \eqref{p1} with fixed $p$}\label{algo1}
\begin{algorithmic}[1]
\STATE {Require:$0\leq x_0\leq1$}
\WHILE{termination criterion $>1e^{-5}$}
      \STATE{ Update $d^k$  by \eqref{d2}}
      \STATE {Update $x^{k+1}$ by \eqref{x1}}
      \STATE{$k:=k+1$}
\ENDWHILE
\RETURN The optimal solution $\boldsymbol{x^{\star}}$.
\end{algorithmic}
\end{algorithm}
\subsection{An acceleration algorithm based on BB method}
While the previously proposed algorithm is straightforward and computationally efficient, it exhibits relatively long computation time for local convergence. Consequently, we opt to use the projection alternating BB method introduced by Dai and Fletcher~\cite{2005Projected} to speed up the computational process.

The projection alternating BB method strategically alternates between utilizing long and short BB step sizes, as outlined below:
\begin{equation}\label{bb1}
\alpha_{k}^{A B B}=\left\{\begin{array}{ll}
\alpha_{k}^{B B 1}, & \text { for odd } k, \\
\alpha_{k}^{B B 2}, & \text { for even } k.
\end{array}\right.
\end{equation}
Here, 
$$
\alpha_{k}^{B B 1}=\frac{s_{k-1}^{T} s_{k-1}}{s_{k-1}^{T} y_{k-1}}
\text { and }
\alpha_{k}^{B B 2}=\frac{s_{k-1}^{T} y_{k-1}}{y_{k-1}^{T} y_{k-1}},
$$
where $s_{k-1}=x_{k}-x_{k-1} $ and  $y_{k-1}=g^{k}-g^{k-1}$. In the context of problem \eqref{p1}, we have $g^k:=\nabla h \left(x^{k},p\right)$. Then the update of $x^{k+1}$ is expressed as 
 \begin{equation}\label{x2}
    x^{k+1}=\mathcal P(x^{k}-\alpha_{k}^{A B B}g^k).
 \end{equation}

We present the framework of this projection BB method in Algorithm~\ref{algo2}.

\begin{algorithm}
\caption{PABB for subproblem \eqref{p1} with fixed $p$}\label{algo2}
\begin{algorithmic}[1]
\STATE{Require: $0\leq x_0\leq1$}
\WHILE{termination criterion $>1e^{-5}$}
      \STATE Update $\alpha_{k}^{A B B}$ by \eqref{bb1}
      \STATE Update $x^{k+1}$ by \eqref{x2}
      \STATE $k:=k+1$
\ENDWHILE 
\RETURN The optimal solution $\boldsymbol{x^{\star}}$.
\end{algorithmic}
\end{algorithm}

\subsection{Termination criterion}\label{kkt}
For the optimization problem defined by \eqref{p1}, the termination criterion is established based on the Karush-Kuhn-Tucker (KKT) conditions.

The Lagrangian function is given by $$L(x,u,v)=x^{T} (Q-P^{\star}) x+(b+p)^T x-u^T x+v^T(x-e^T),$$
where $u$ and $v$ correspond to the dual variables. Then the first-order optimality conditions can be expressed as follows:
\begin{equation}\label{kkt1}
 \begin{aligned}
 \nabla_x L=2  (Q-P^{\star}) x+&(b+p)^T-u^T+v^T=0, \\
0\leq x& \leq 1,\\
u^Tx&=0, \\
v^T(e-x)&=0, \\
u\geq 0&,v\geq 0.\\
\end{aligned}
\end{equation}
By introducing the index sets $I=\{i|0<x_i<1\}$, $K=\{i|x_i=1\}$, $J=\{i|x_i=0\}$, and defining $\bar{Q}=Q-P^{\star}$ and $\bar{b}=b+p$, the conditions can be rewritten as:
$$
\begin{array}{l}
cret_1=2\bar{Q}_{II} x_{I}+\bar{Q}_{Ik}e_K+\bar{b}_{I}=0, \\
cret_2=2\bar{Q}_{JI}x_{I}+\bar{Q}_{JK}e_K+\bar{b}_{J}=u^T_{J}\geq 0, \\
cret_3=2\bar{Q}_{KI} x_{I}+\bar{Q}_{KK}e_K+\bar{b}_{K}=-v^T_{K}\leq0.
\end{array}
$$

Consequently, we present the termination criterion as 
\begin{equation}\label{cret}
    cret/\rho<\delta,
\end{equation}
where
$$cret=max\{max(|cret_1|),max\{|min(cret_2,0)|\},max\{|max(cret_3,0)|\}\}$$
$$\text{and }\rho=max\{1,|2(Q-P^{\star})|,|b+p|\}.$$

\subsection{Penalty parameter update}
Due to the adoption of the penalty factor update from PSDP, specifically $P^\star = P^\star + \alpha Z$, where both $P^\star$ and $Z$ are diagonal matrices in the vectorized PSDP algorithm, only the diagonal elements of $P^\star$ need to be updated, that is,
\begin{equation}\label{pp1}
    p^{s}=p^{s-1}+\alpha^s z^s.
\end{equation}
 Here, $z^s=(x_1^s-x_1^{s^2},x_2^s-x_2^{s^2},\dots, x_n-x_n^{s^2})^T$, $x^s$ denotes the best value of $x$ after the $s-1$ th update of $p$, and $p^s$ denotes the penalty parameter after s updates.

The choice of the parameter $\alpha$ significantly impacts the overall efficiency of the algorithm. When $\alpha$ is too small, it may lead to smaller iteration steps and an increase in the number of iterations, thus increasing the runtime and slowing down convergence. Conversely, if $\alpha$ is too large, it can easily cause $x$ to rapidly approach the boundaries during the update process, potentially halting at a low-quality local solution. Therefore, selecting an appropriate value for $\alpha$ is crucial.

We opt for an appropriate $\alpha$ from the perspective of the active set. Ideally, after updating $p$, we desire the new optimal solution $x$ of \eqref{p1} to have a different active set compared to before the update, while ensuring that the change is not too drastic. In the following lemma, we provide an $\alpha$ that guarantees a change in the active set. 
\begin{lemma}
      In the iterations, we choose $\alpha^s$ as (\ref{eq:alpha}), then the active set can be updated. 
      \begin{equation}\label{eq:alpha}
    \alpha^s=\lambda_{min}(z^{s^{-\frac{1}{2}}}\bar{Q}_{II}^sz^{s^{-\frac{1}{2}}}).
\end{equation}
\end{lemma}
\begin{proof}

    From the KKT condition, we have 
\begin{equation*}
    2\bar{Q}_{II} x_{I}+\bar{Q}_{IK}e_K+\bar{b}_{I}=0,
\end{equation*}
let $y_I=2x_I-e_I$, then we have 
\begin{equation*}
    \bar{Q}_{II} y_{I}+\bar{Q}_{IK}e_K+\bar{Q}_{II}e_I+\bar{b}_{I}=0,
\end{equation*}

where $\bar{Q}_{II}e_I+\bar{b}_{I}$ is constant for fixed I. Thus we have $\bar{Q}_{II}^s y_{I}^s=\bar{Q}_{II}^{s+1} y_{I}^{s+1}$ for fixed I, and $y_{I}^s$, $\bar{Q}_{II}^s$ means the value of $y_{I}$, $\bar{Q}_{II}^s$ after the sth update of $p$. We want $x_I^{s+1}\leq 0$ or $x_I^{s+1}\geq 1$, which means $|2x_{I}^{s+1}-e_I|\geq 1$. We update $p$ by \eqref{pp1}, then $\bar{Q}_{II}^{s+1}=\bar{Q}_{II}^s-\alpha^s Z_{II}^s$, Z is the diagonal matrix with z as the diagonal element. We have 
$$(\bar{Q}_{II}^s-\alpha^s Z_{II}^s)(2x_{I}^{s+1}-e_I)=\bar{Q}_{II}^s(2x_{I}^s-e_I). $$ Easy to see, $2x_{I}^{s+1}-e_I \rightarrow \infty $ may happens when $\lambda_{min} (\bar{Q}_{II}^s-\alpha^s Z_{II}^s) \longrightarrow 0 $, then the active set can be updated. Therefore we take 
$$
    \alpha^s=\lambda_{min}(z^{s^{-\frac{1}{2}}}\bar{Q}_{II}^sz^{s^{-\frac{1}{2}}}).
$$
While \eqref{eq:alpha} may not be the smallest $\alpha^s$ that induces a change in the active set, we apply a slight scaling factor of $0<\eta<1$ to $\alpha^s$, which is  
 \begin{equation}\label{eq:alpha1}
    \alpha^s=\eta \cdot \lambda_{min}(z^{s^{-\frac{1}{2}}}\bar{Q}_{II}^sz^{s^{-\frac{1}{2}}}).
\end{equation}
The above constitutes our proof.\qed
\end{proof}

\subsection{Initialization}
Within this section, we introduce two initialization techniques: reformation and random perturbation.

\textbf{Reformation.}
Without loss of generality, we initialize the vectorized PSDP by setting $p=0$. For the UBQP problem \eqref{bqp}, leveraging the binary nature of $x \in \{0,1\}^n$, we can reformulate the objective in an equivalent form by introducing an arbitrary vector $\gamma$, yielding
$$
x^{T} (Q+Diag(\gamma)) x + (b-\gamma)^T x.
$$
Although the new $Q$ (i.e. $Q+Diag(\gamma)$) and $b$ (i.e. $b-\gamma$) do not change the solution of problem \eqref{bqp}, the solution of problem \eqref{p1} subject to the constraint $x \in [0,1]^n$ is affected by $\gamma$. Therefore, $\gamma$ will affect the initial solution of $x$ in our algorithm.

As noted in the previous subsection, it is imperative to prevent the iterative point sequence from approaching the boundary too rapidly, as this could impede the discovery of high-quality solutions. Consequently, the selection of an appropriate $\gamma$ is crucial to ensure that the initial solution for $x$ resides within the interior of the feasible region. Notably, this goal can be achieved by setting $\gamma$ sufficiently large. This is due to the inactive nature of the boundary defined by $0\leq x_i\leq 1$, where the first-order condition is satisfied as
\begin{equation}\label{gamma}
   2(Q+Diag(\gamma))x=\gamma-b. 
\end{equation}
Analysis of \eqref{gamma} reveals that as $\gamma$ approaches infinity, $x_i$ converges to $1/2$. To maintain the constraint $0<x_i<1$, an appropriately large $\gamma$ can be chosen. For example, setting $\gamma_i > \sum_{j=1}^{n} 2|Q_{ij}| + |b_{i}|$ for $i=1,\ldots,n$ in the selection of $\gamma$ ensures $0< x_i< 1$, as can be deduced from \eqref{gamma}.

\textbf{Random Perturbation.}
For integer problems (where the elements of $Q$ and $b$ are integers), certain components of the algorithm may occasionally stall around 1/2. Therefore, following the approach of Tang and Toh~\cite{tang2023feasible}, we introduce a symmetric random perturbation from a normal distribution.
\begin{equation}
    Q_{ij}=Q_{ij}+\mu_{ij},
\end{equation}
where $\mu_{ij} \sim \mathbb{N}  (0,\sigma)$ and $\mu_{ij}=\mu_{ji}$.

Due to the continuity of problem \eqref{p1}, small perturbations do not alter the optimal solution. By incorporating random perturbations, the iterations no longer stagnate around 1/2, thereby enhancing the robustness and stability of the algorithm.

Having outlined all the details of the vectorized PSDP algorithm, we now present its framework in Algorithm~\ref{algo3}.
\begin{algorithm}
\caption{Vectorized PSDP for UBQP}\label{algo3}
\begin{algorithmic}[1]
\STATE \textbf{Initialization:} 
\STATE Select a sufficiently large $\gamma$, then update $Q:=Q+Diag(\gamma)$ and $b:=b-\gamma$
\STATE Solve $x^0$ from \eqref{gamma}, and set $p^0=0$ and $s=0$
\WHILE{$\Vert \mathcal P(x^s)-x^s \Vert_\infty>1e^{-5}$}
      \STATE Update $p^{s+1}$ by \eqref{pp1} and \eqref{eq:alpha}
      \STATE Solve problem \eqref{p1} with $p=p^s$ and $x_0=x^s$ using either Algorithm~\ref{algo1} or Algorithm~\ref{algo2} to obtain the optimal solution $x^{s+1}$ based on the stopping criterion \eqref{cret}
      \STATE $s:=s+1$
\ENDWHILE 
\RETURN {The optimal solution $\boldsymbol{x^{\star}}$}
\end{algorithmic}
\end{algorithm}

\section{Theoretical analysis}\label{sec4}
\subsection{Equivalence of the models}
In this subsection, we prove that the penalty problem \eqref{p1} is exact, which means that its optimal solution is also an optimal solution of problem \eqref{o1}, when all the components of $p$ are larger than a certain constant.

To simplify the description of the problem, we rewrite the problem \eqref{o1} as
\begin{equation}\label{o}
    \begin{array}{cl}
\min _{x} & f(x) \\
\text { s.t. } & x \in F .
\end{array}
\end{equation}
where $f(x)=x^{T} Q x+b^{T} x$ and $F=\{x|x \in \{0,1\}^n\}$.
Also we rewrite the problem \eqref{p1} as 
\begin{equation}\label{p}
\begin{array}{cl}
    \min _{x } &h(x, p) \\
    \text { s.t. } & x \in G .
 \end{array}   
\end{equation}
where  $h(x,p):=f(x)+\sum_{i=1}^{n}p_i(x_i-x_i^2) $ and $G=\{x|0\leq x \leq 1\}$.

For the convenience of our proof, let us first propose two lemmas.
\begin{lemma}\label{lemma1}
$$
    \min _{ x \in G} h\left(x,p_{1}\right) \leq \min _{x \in G} h \left(x, p_{2}\right) \leq \min _{x \in F} f(x) ,\forall p_{1} \leq p_{2} .
$$
Proof. For any p we have
$$
\min _{x \in G} h(x,p) \leq \min _{x \in F} h(x, p)=\min _{x \in F} f(x) .
$$
For any  $p_{1} \leq p_{2}$  we have  $p_{2}-p_{1}\leq 0$ , which implies  $\left(p_{2}-p_{1}\right)_i   (x_i-x_i^2) \geq 0$  for all  $ x\in G $ . Consequently, $p_{2i}  (x_i-x_i^2) \geq p_{1i}  (x_i-x_i^2) $ . Thus we have
$$
h \left(x, p_{1}\right) \leq h\left(x, p_{2}\right) ,\quad \forall x\in G .
$$
Then
$$
\min _{x \in G} h\left(x, p_{1}\right) \leq \min _{x \in G} h\left(x,p_{2}\right),
$$
which completes our proof. \qed
\end{lemma}
\begin{lemma}\label{lemma2}
    If $x^{\star}$ is a local (global) optimal point of problem \eqref{p}, then $x^{\star} \in F$, for any sufficiently large p.
\end{lemma}
\begin{proof} 
    If $x^{\star} \notin F$, there at least one $x_i$ satisfies  $0 < x_i < 1$, then exists a sequence feasible direction $d = e_i $ leads to  
\[  
d^T \nabla^2 L(x,u,v)  d = d^T (Q - P^{\star}) d = Q_{ii} - p_i < 0 ,\text{ for } p_i > Q_{ii},  
\]  
where $L(x,u,v)=x^T(Q-P^{\star})x+(b+p)^Tx-u^T x+v^T(x-e)$  is the Lagrangian function function of problem \eqref{p}.
We can observe that it is in contradiction with the second-order necessary condition.\qed
\end{proof}

By leveraging Lemma~\ref{lemma1} and \ref{lemma2}, we can determine the exact penalty of problem \eqref{p}, as demonstrated in Theorem~\ref{thm:3}.

\begin{theorem}\label{thm:3}
If $\bar{x}$ is the optimal point of problem $\eqref{o}$, it is the global optimal point of problem \eqref{p} when $p>\bar{p}$; If $x^{\star} $ is the global optimal point of problem \eqref{p},then it is the optimal point of problem $\eqref{o}$ when $p>\bar{p}$, where $\bar{p}=(Q_{11}, Q_{22},\dots, Q_{nn})$.
\end{theorem}
\begin{proof}
We have
    $$f(\bar{x})=\underset{x\in F}{\min }  f(x)=\underset{x\in F}{\min }  h(x,p)=\underset{x\in G}{ \min }  h(x,p)=h(x^{\star},p),\text{ for any }p>\bar{p}.
$$
The left equality is from Lemma~\ref{lemma1} and the right one is from Lemma~\ref{lemma2}.\qed
\end{proof}
\subsection{Convergence of the algorithm}
In Algorithm~\ref{algo1}, we apply a proximal point algorithm to solve the subproblem \eqref{p1}, then we present the convergence of the Algorithm~\ref{algo1} in Lemma~\ref{lemma3} and Theorem~\ref{thm2}.

\begin{lemma}\label{lemma3}
Suppose that  $\left\{x^{k}\right\}$ are generated by Algorithm~\ref{algo1}, then
$$
h\left(x^{k+1}, p\right) \leq h\left(x^{k},  p\right)-d^{T} (H_p-Q+P^{\star})d .
$$
\end{lemma}
\begin{proof}
 Since for all $k$ we have $x^{k}\in F_1$, it can be obtained that
$$
\begin{aligned}
h\left(x^{k+1},  p\right) & =\left(h\left(x^{k}+d^{k}, p\right)+d^{k^T} (H_p-Q+P^{\star}) d^{k}\right)-d^{k^T}   (H_p-Q+P^{\star}) d^{k} \\
& \leq h\left(x^{k}, p\right)-d^{k^T}   (H_p-Q+P^{\star}) d^{k} .
\end{aligned}
$$
Here, the inequality holds due to \eqref{dp}.\qed
\end{proof}

\begin{theorem}\label{thm2}
 Suppose that  $\left\{d^{k}\right\},\left\{x^{k}\right\}$  are generated by Algorithm~\ref{algo1}. Since our feasible region is bounded are bounded, then  $\left\{d^{k}\right\} $ converges to 0 and any cluster point of  $\left\{\left(x^{k}, d^k,u^{k}, v^{k}\right)\right\}$  is a first-order stationary point of problem \eqref{p1}. 
\end{theorem}
\begin{proof}
    
Based on Lemma~\ref{lemma3}, it is established that
$$
h\left(x^{k+1}, p\right) \leq h\left(x^{0}, p\right)-\sum_{i=1}^{k} d^{i^{T}} (H_p-Q+P^{\star}) d^{i} .
$$
Consequently, due to the boundness of function $h$, we derive
$$
\sum_{k=1}^{\infty} d^{k^{T}}(H_p-Q+P^{\star})   d^{k}<+\infty.
$$
Hence, $d^{k} \rightarrow 0$ as $H_p-Q+P^{\star} \succ 0$.

Returning to our subproblem \eqref{dp}, the KKT conditions are outlined below:
\begin{equation}\label{kkt2}
 \begin{array}{c}
 2(Q-P^{\star})x^k+(b+p)+2H_pd^k-u^k+v^k=0, \\
 u^k,v^k\leq 0, \\
u_k^T(x^k+d^k)=0, \\
v_k^T(e-x^k-d^k)=0, \\
0\leq x^k+d^k\leq 1.\\
\end{array}
\end{equation}
We can see that, for $x_I^{k+1}$, we have $v^k_I=u^k_I=0$; for $x_K^{k+1}$, $u^k_K=0$; for $x_J^{k+1}$, $v^k_J=0$. Thus, $u^{k^T}v=0$, then we have
 $\left\|u^{k}\right\|_{\infty} \leq 2\left\|Q-P^{\star}\right\|_{\infty}+\|b+p\|_{\infty}+\|H p\|_{\infty}$, $
\left\|v^{k}\right\|_{\infty} \leq 2\left\|Q-p^{\star}\right\|_{\infty}+\|b+p\|_{\infty}+\|H p\|_{\infty}$, indicating that $\{(x^k,d^k,u^k,v^k)\}$ is bounded and possesses a convergent subsequence. Assume that $ \{(x^{i_k},d^{i_k},u^{i_k},v^{i_k})\}$ is a subsequence that converges to an arbitrary cluster point $(x^{\star},d^{\star},u^{\star},v^{\star})$. With $d^{k} \rightarrow 0$, it follows that $d^{\star}= 0$. Let $i_{k} \rightarrow \infty$ in \eqref{kkt2}, then $(x^{\star},d^{\star},u^{\star},v^{\star})$ satisfies the KKT conditions \eqref{kkt1}, which implies that it is a first-order stationary point of problem \eqref{p1}.\qed
\end{proof}

In Algorithm~\ref{algo2}, an alternative BB method is adopted, which is non-monotonic, allowing for instances where $h\left(x^{k+1},  p\right)$ may increase during certain iterations. Dai and Fletcher~\cite{2005Projected} provided counterexamples illustrating that the Projected BB and Projected Alternating BB methods, without line searches, could cyclically oscillate between multiple points, and such behavior is atypical. Furthermore, Li and Huang~\cite{li2023note} theoretically proved that the alternative BB method without a monotone line search converges at least R-linearly with a rate of $1-1 / \kappa$ under certain properties for convex unconstrained quadratic problems. However, we have not yet found a relevant theoretical result for such a non-convex problem. In the subsequent section, numerical experiments will be conducted to demonstrate that the PABB method does not exhibit cycling phenomena.

\section{Numerical experiment}\label{sec5}

In this section, we test our proposed algorithms on some UBQP benchmark datasets and random generated datasets. Since our method is derived from the SDP, we compare it with the SDR approach based on the SDPT3 solver~\cite{R2003Solving}, which is used to solve subproblems in~\cite{luo2019enhancing,RePEc:spr:jcomop:v:39:y:2020:i:3:d:10.1007_s10878-019-00517-8}. Furthermore, we contrast our approach with two similar penalty function-based methods~\cite{RePEc:spr:jcomop:v:39:y:2020:i:3:d:10.1007_s10878-019-00517-8,yuan2017exact}, referred to as ALM and MPEC in this paper.
The specific parameters for the five algorithms are outlined as follows:
\begin{itemize}
    \item PPA (Vectorized PSDP method with proximal point algorithm) \\ Algorithm~\ref{algo1} is utilized for solving the subproblems, employing a termination criterion of $10^{-5}$ or a maximum of 10000 iterations. The penalty function is updated for a maximum of 1000 iterations.
    \item PABB (Vectorized PSDP method with projection alternative BB algorithm)\\ Algorithm~\ref{algo2} is employed to address the subproblems, utilizing a termination criterion of $10^{-5}$ or a maximum of 10000 iterations. The penalty function is updated for a maximum of 1000 iterations.
    \item SDR (Semidefinite relaxation)\\
    It is implemented using the SDPT3 solver, with all termination criteria set to $10^{-5}$.
    \item MPEC (Exact penalty method for binary optimization based on MPEC formulation)\\
    All the parameter settings remain consistent with the code provided by Yuan\footnote{https://yuangzh.github.io/}. 
    \item ALM (The augmented Lagrangian algorithms based on continuous relaxation)\\
    All the parameter settings remain consistent with the original papers. We consider $x_0 = pe$, where $p$ is a randomly chosen scalar in the interval $(0,1)$ and $e$ is the vector of all ones.
\end{itemize}

We test our algorithms on a total of six datasets, which are Random Instances, QPLIB\footnote{https://qplib.zib.de/}, BE\footnote{https://biqmac.aau.at/biqmaclib.html} (Billionnet and Elloumi instances), GKA\footnote{https://biqmac.aau.at/biqmaclib.html} (Glover, Kochenberger and Alidaee instances), rudy\footnote{http://www-user.tu-chemnitz.de/$\sim$helmberg/rudy.tar.gz} and the G set\footnote{https://web.stanford.edu/$\sim$yyye/yyye/Gset/}. We evaluate the computational efficiency and solution quality using CPU time and the solution GAP as metrics, where the GAP calculation formula is defined as:
$$\text { GAP }=\left|\frac{\mathrm{obj}-\mathrm{lb}}{\mathrm{lb}}\right| \times 100 \%.$$
where the lower bound ($\mathrm{lb}$) is determined through the utilization of the Biq Mac Solver\footnote{https://biqmac.aau.at/}.

\subsection{Convergence verification}

Due to the absence of any line search, it is challenging to provide the theoretical convergence analysis for Algorithm \ref{algo2}. Here, we validate its practical convergence through numerical experiments. We present one instance from each of the six datasets, and Fig.\ref{fig5} illustrates the number of iterations used by Algorithm \ref{algo2} in solving the subproblems. It can be observed from the figure that the number of iterations for the subproblems is mostly within 1000. For the instance with the QPLIB ID 3565, there is a small probability that the number of iterations for the subproblems exceeds 2000. However, even in such cases, all iteration counts remain below 6000, not reaching the maximum of 10000 iterations set for the experiments. This indicates to some extent that the PABB method without line search does not exhibit cycling, confirming the convergence of Algorithm \ref{algo2} in practice.

Next, we validate the overall convergence of the two vectorized PSDP algorithms (PPA and PABB) in numerical experiments. Fig.\ref{fig4} illustrates the convergence process of the two algorithms in reducing the number of components violating the 0-1 constraints over six test cases. While not strictly monotonically decreasing, the overall trend is toward reduction and convergence to zero. This indicates that our algorithms ultimately converge to feasible points satisfying the 0-1 constraints. 

\begin{figure}[htbp]  
    \centering 
 
    \begin{minipage}[c]{0.33\linewidth}  
        \vspace{3pt}  
        \centering 
        \includegraphics[width=\textwidth]{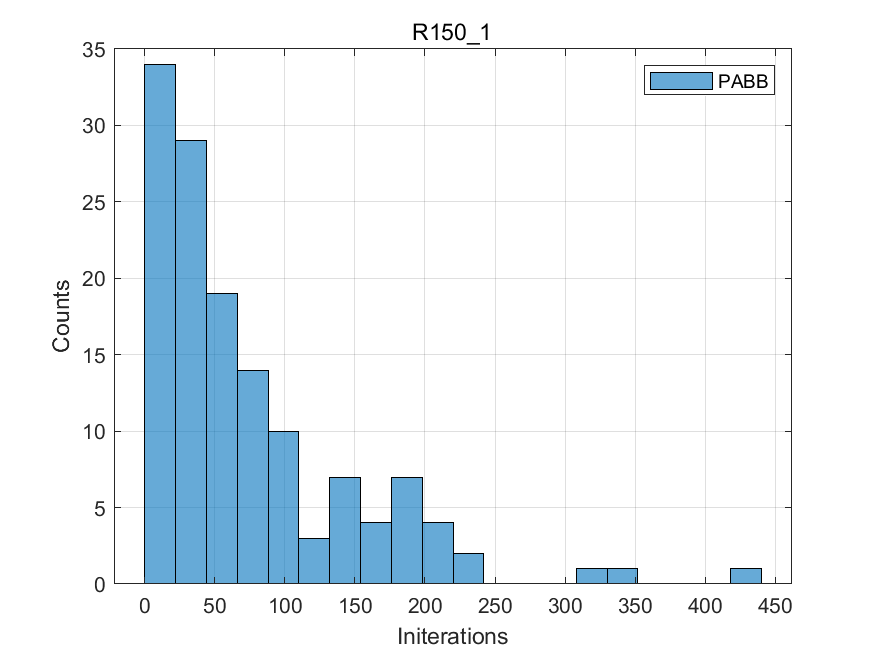}  
        (a) Random datasets  
    \end{minipage}  
    \hfill 
    \begin{minipage}[c]{0.325\linewidth}  
        \vspace{3pt}  
        \centering  
        \includegraphics[width=\textwidth]{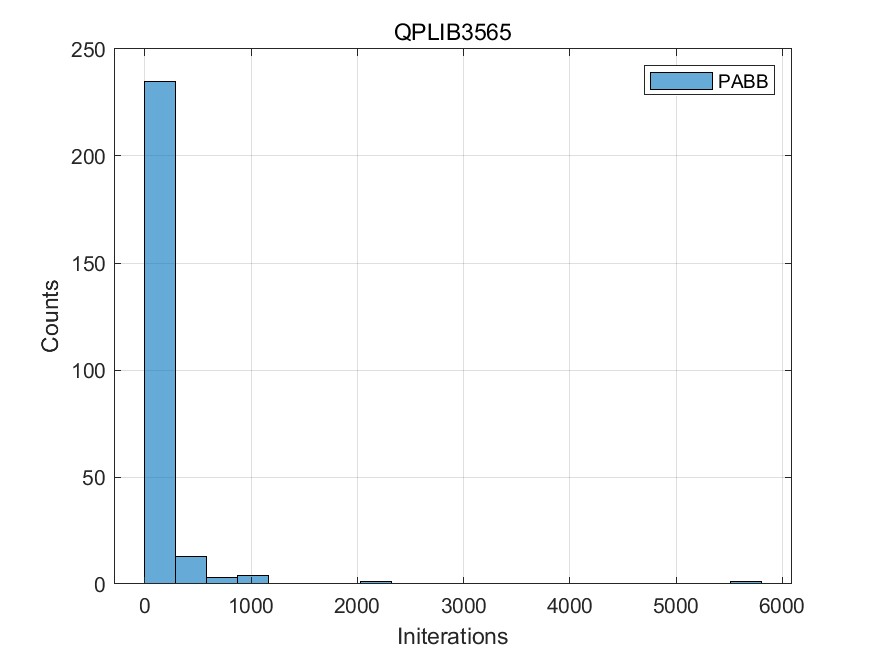}  
        (b) QPLIB datasets  
    \end{minipage}  
    \hfill  
    \begin{minipage}[c]{0.325\linewidth}  
        \vspace{3pt}  
        \centering  
        \includegraphics[width=\textwidth]{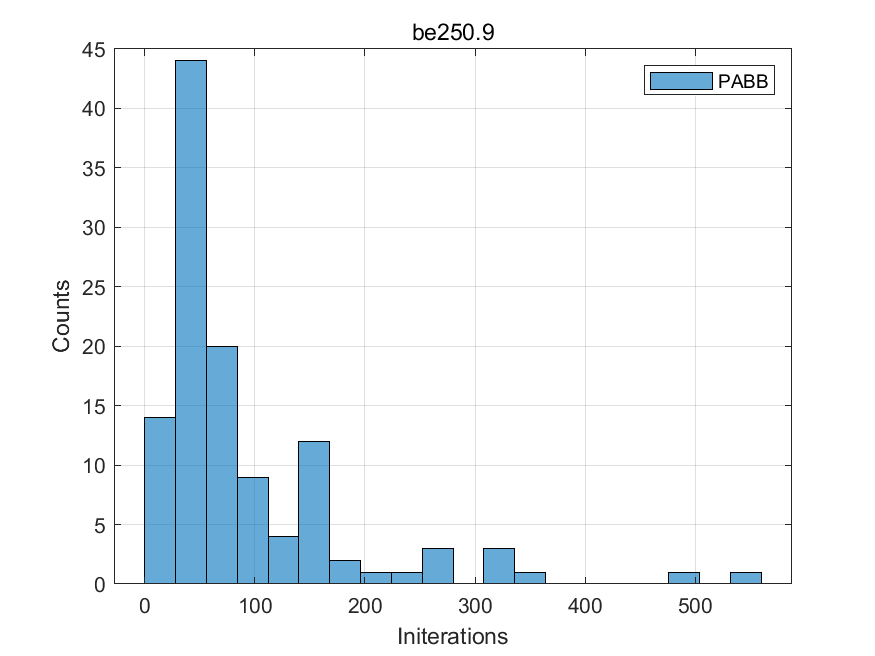}  
        (c) BE datasets  
    \end{minipage}  
\hfill
    \begin{minipage}[c]{0.33\linewidth}  
        \vspace{3pt}  
        \centering  
        \includegraphics[width=\textwidth]{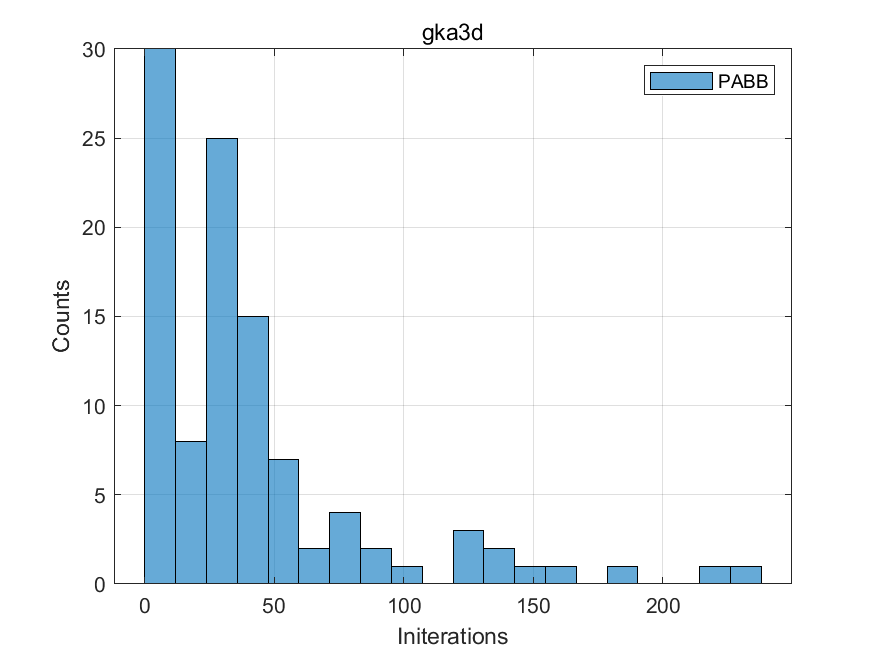}  
        (d) GKA datasets  
    \end{minipage}  
    \hfill  
    \begin{minipage}[c]{0.33\linewidth}  
        \vspace{3pt}  
        \centering  
        \includegraphics[width=\textwidth]{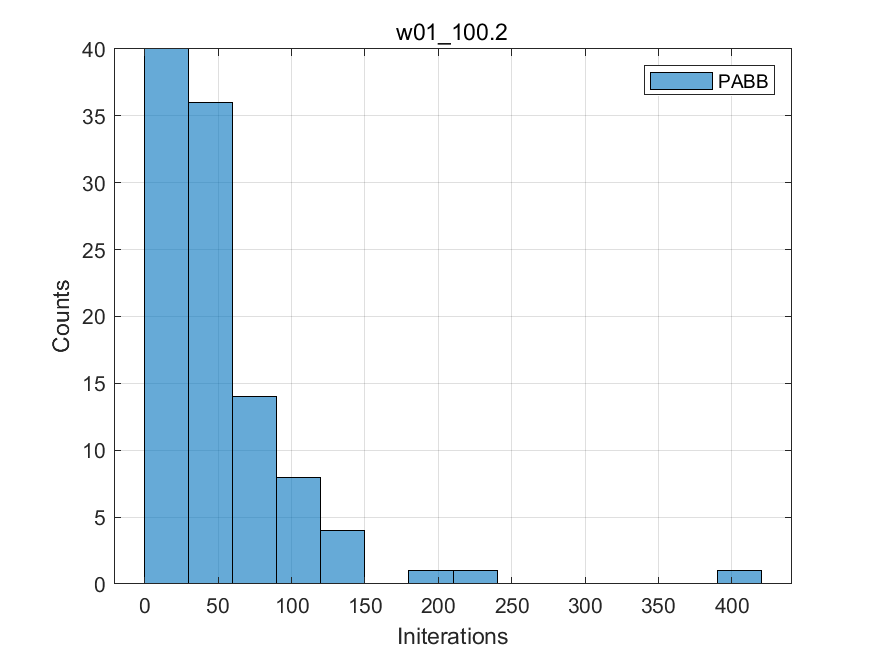}  
        (e) Rudy datasets  
    \end{minipage}  
    \hfill  
    \begin{minipage}[c]{0.32\linewidth}  
        \vspace{3pt}  
        \centering  
        \includegraphics[width=\textwidth]{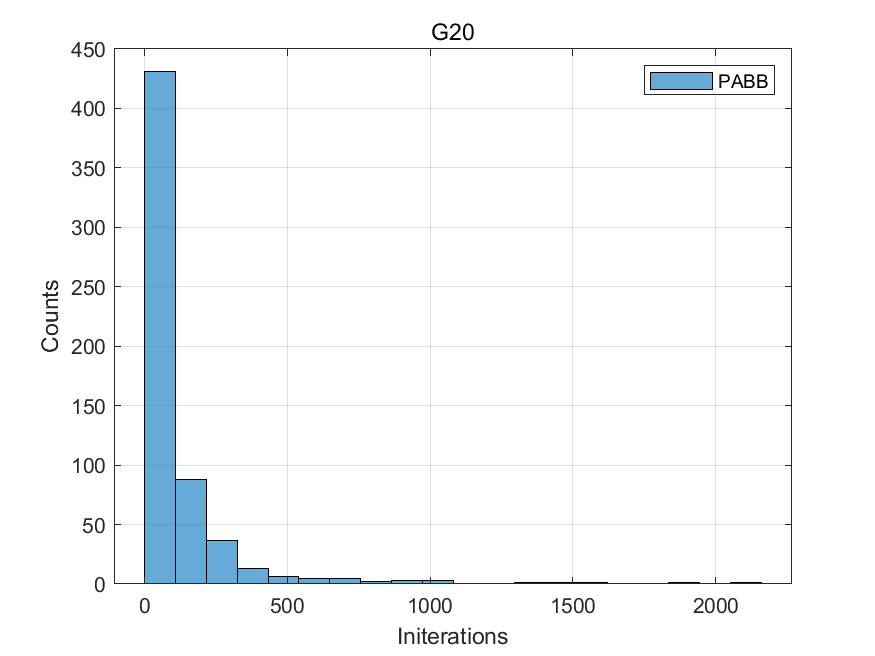}  
        (f) G sets  
    \end{minipage}  
    \caption{Histogram of the number of iterations for subproblems }  
    \label{fig5}  
\end{figure}

\begin{figure}[htbp]  
    \centering 

    \begin{minipage}[c]{0.33\linewidth}  
        \vspace{3pt}  
        \centering 
        \includegraphics[width=\textwidth]{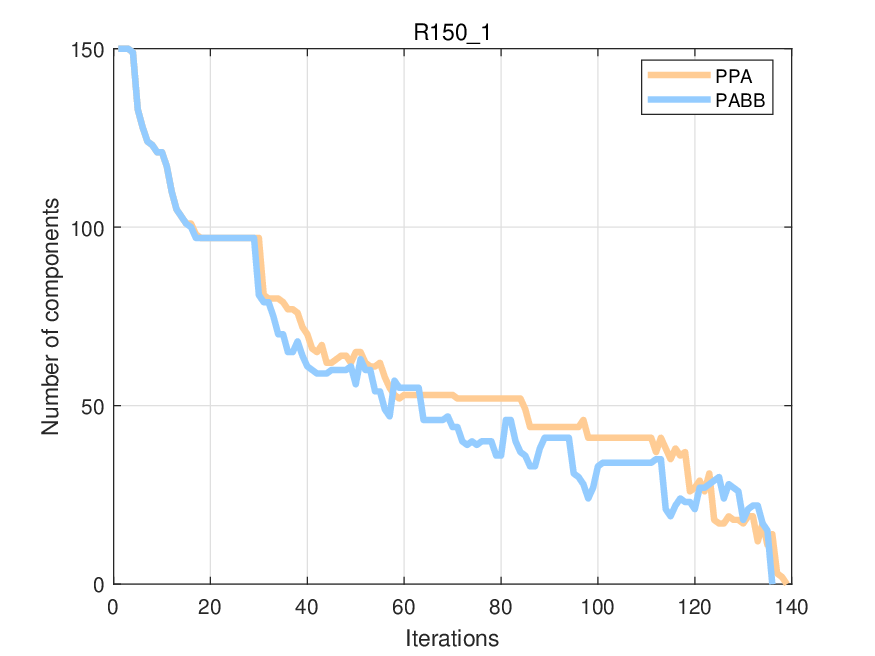}  
        (a) Random datasets  
    \end{minipage}  
    \hfill 
    \begin{minipage}[c]{0.325\linewidth}  
        \vspace{3pt}  
        \centering  
        \includegraphics[width=\textwidth]{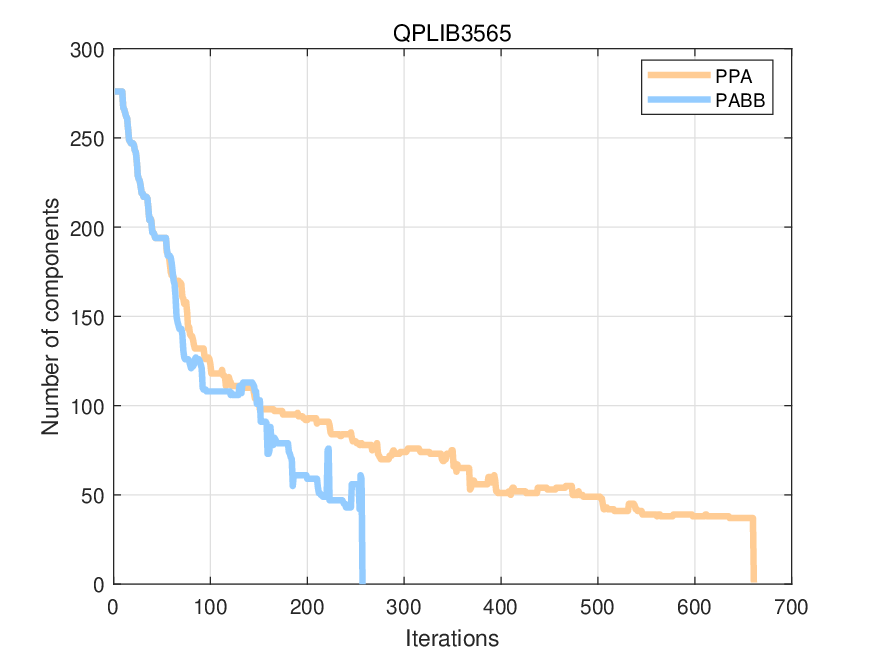}  
        (b) QPLIB datasets  
    \end{minipage}  
    \hfill  
    \begin{minipage}[c]{0.325\linewidth}  
        \vspace{3pt}  
        \centering  
        \includegraphics[width=\textwidth]{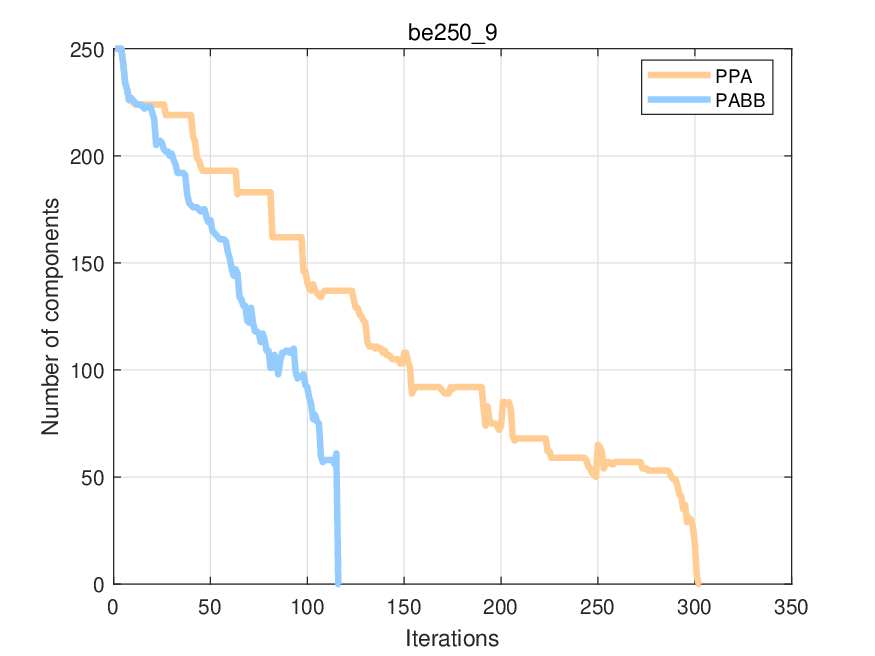}  
        (c) BE datasets  
    \end{minipage}  
\hfill
    \begin{minipage}[c]{0.33\linewidth}  
        \vspace{3pt}  
        \centering  
        \includegraphics[width=\textwidth]{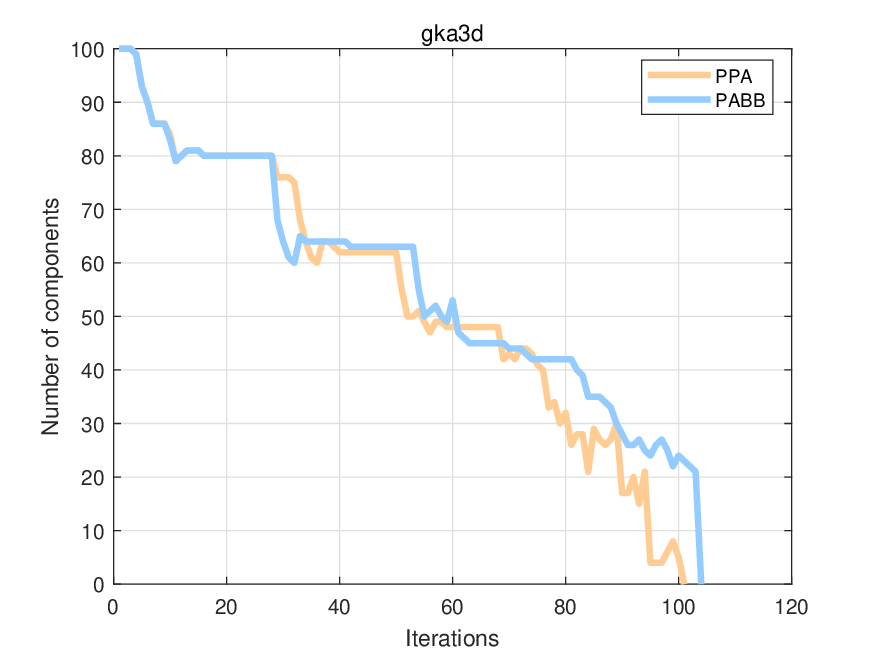}  
        (d) GKA datasets  
    \end{minipage}  
    \hfill  
    \begin{minipage}[c]{0.33\linewidth}  
        \vspace{3pt}  
        \centering  
        \includegraphics[width=\textwidth]{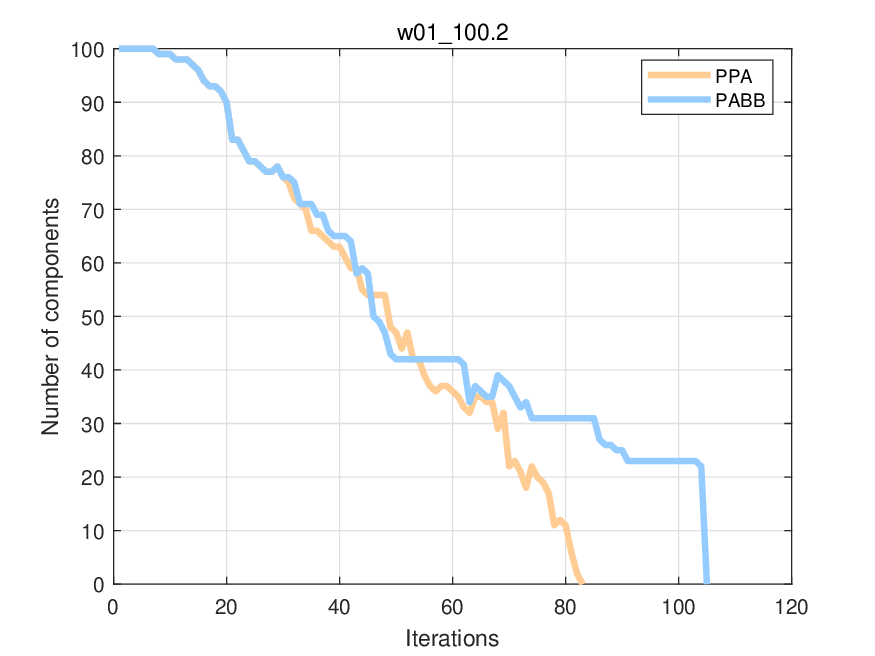}  
        (e) Rudy datasets  
    \end{minipage}  
    \hfill  
    \begin{minipage}[c]{0.32\linewidth}  
        \vspace{3pt}  
        \centering  
        \includegraphics[width=\textwidth]{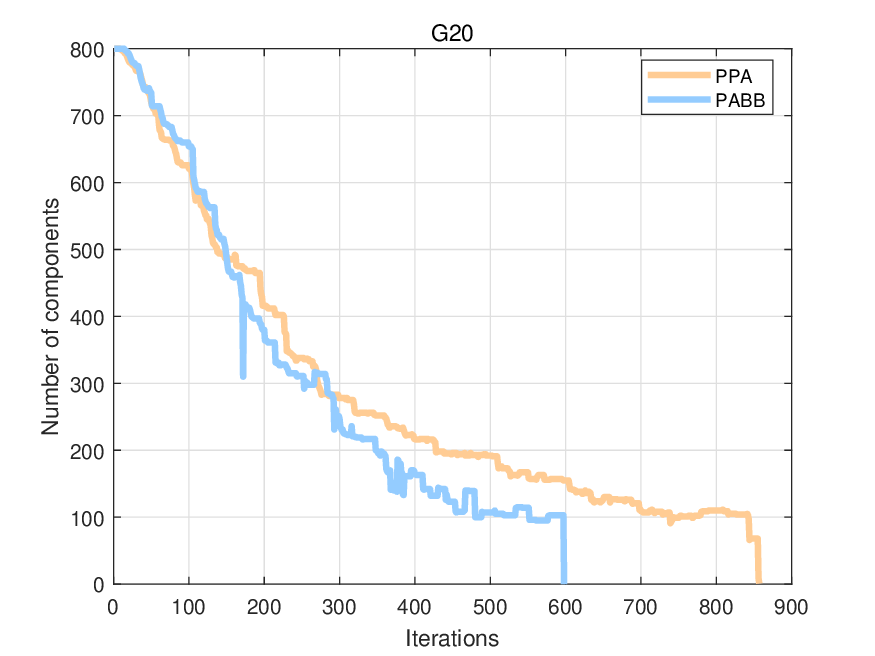}  
        (f) G sets  
    \end{minipage}  
    \caption{The number of variables dissatisfied the constraint conditions }  
    \label{fig4}  
\end{figure}

\subsection{0-1  quadratic  problems}
The UBQP problem to be addressed is formulated as follows:
$$
\min \left\{y^{T} Q y: y \in\{0,1\}^{n}\right\},
$$
where  Q  is a symmetric matrix of order $n$.
For this problem, we conducted experiments on four datasets, namely: random instances, QPLIB, BE, and GKA.

\textbf{(Random Instances)} We randomly generate the symmetric matrix $Q$ with dimensions of 50, 100, 150, 200, and 250. For each dimension, we create five test cases. Due to the stochastic nature of some algorithms, we run each algorithm 10 times for each test case and report the average GAP and CPU time. In cases where the runtime of the ALM algorithm exceeds 100 seconds or yields a non-optimal trivial solution of 0, we classify such instances as failures.

The detailed results are summarized in Table~\ref{t1}. We observe that the five algorithms do not exhibit significant differences in terms of runtime, with the SDR algorithm running relatively slower. In terms of solution quality, PPA achieves the smallest GAP with the lower bound, closely followed by PABB. However, ALM fails on average 1 to 4 times in each test case.

\begin{table}[htbp]
    \centering
    \caption{Experimental results of Random Instances in dimensions of 50, 100, 150, 200, and 250 }
    \label{t1}
    \setlength\tabcolsep{0.9pt}
    \begin{tabular*}{\textwidth}{@{\extracolsep\fill}lccccccccccl}
\toprule  
~& \multicolumn{5}{c}{CPU Time(s)} & \multicolumn{5}{c}{GAP(\%)} &Fall \\ \cmidrule{2-6}\cmidrule{7-11}%
n & PPA & PABB & ALM &MPEC & SDR & PPA & PABB & ALM&MPEC & SDR  &  ALM \\
\midrule
   50 & 0.13 & \textbf{0.05} & 0.13 & 0.33 & 0.42 & \textbf{0.51} & 0.53 & 2.09 & 0.80 & 2.19 & 4.40 \\ 
        100 & 0.38 &\textbf{ 0.10} & 0.20 & 0.23 & 0.44 & 0.15 & \textbf{0.07} & 2.63 & 0.95 & 3.56 & 3.00 \\ 
        150 & 1.14 & \textbf{0.20} & 0.30 & 0.26 & 0.56 & \textbf{0.21} & 0.23 & 2.72 & 0.79 & 2.98 & 1.80 \\ 
        200 & 3.21 & 0.46 & 0.46 & \textbf{0.29} & 0.86 & \textbf{0.14} & 0.21 & 2.04 & 0.82 & 4.33 & 1.80 \\ 
        250 & 5.07 & 0.65 & 0.63 & \textbf{0.34} & 1.17 & \textbf{-0.15} & -0.12 & 1.52 & 0.50 & 3.20 & 1.60 \\ 
\hline
\end{tabular*}
\end{table}

Additionally, we consider some lower-dimensional matrices with dimensions of 20, 40, 60, and 80, and apply the branch-and-bound method to obtain the optimal solution as the lower bound (lb). As shown in Table~\ref{t2}, PABB is relatively faster in terms of time, with a smaller GAP compared to the optimal value, and even the maximum GAP is only 1.02\%. PPA is more accurate, achieving an average GAP of just 0.02\% when solving examples with 80 dimensions. In contrast, ALM has a probability of failure during the solving process, and even when successful, the GAP is relatively high. The MPEC solution is also accurate, but with longer run times. On the other hand, the SDR algorithm has slower solving speeds and relatively lower solution quality.

\begin{table}[htbp]
    \centering
    \caption{Experimental results of Random Instances in dimensions of 20, 40, 60 and 80}
    \label{t2}
    \setlength\tabcolsep{0.9pt}
    \begin{tabular*}{\textwidth}{@{\extracolsep\fill}lccccccccccl}
\toprule  
~& \multicolumn{5}{c}{CPU Time(s)} & \multicolumn{5}{c}{GAP(\%)} &Fall \\ \cmidrule{2-6}\cmidrule{7-11}%
n & PPA & PABB & ALM &MPEC & SDR & PPA & PABB & ALM&MPEC & SDR  &  ALM \\
\midrule
   20 & 0.03 & \textbf{0.02} & 0.09 & 0.34 & 0.47 & 0.09 & \textbf{0.06} & 0.60 & 0.09 & 1.31 & 3.20 \\ 
        40 & 0.11 & \textbf{0.05} & 0.13 & 0.37 & 0.44 & \textbf{1.00} & 1.02 & 2.20 & 1.48 & 4.87 & 4.60 \\ 
        60 & 0.17 & \textbf{0.07} & 0.17 & 0.44 & 0.52 & \textbf{0.09} & 0.12 & 2.75 & 0.35 & 1.96 & 1.40 \\ 
        80 & 0.45 & \textbf{0.15} & 0.41 & 0.42 & 0.66 & \textbf{0.02} & 0.13 & 1.77 & 0.95 & 2.41 & 1.40 \\  
\hline
\end{tabular*}
\end{table}

This observation suggests that the two vectorized PSDP methods have a clear advantage in solution quality over other approaches. PABB generally outperforms PPA in speed, effectively speeding up the process.

\textbf{(QPLIB)}
The QPLIB is a library of quadratic programming instances, which contains 319 discrete and 134 continuous instances of different characteristics. We conduct testing on all binary instances, totaling 23 cases, with dimensions ranging from 120 to 1225 in the QPLIB dataset.

As shown in Table~\ref{t3}, MPEC emerges as the fastest approach, with an average of 0.49 seconds per instance. PABB and ALM follow closely with an average of about 1 second per instance. PPA and SDR take longer. In terms of solution quality, PPA and PABB maintain their superiority.

\begin{table}[htbp]
    \centering
    \caption{Experimental results of QPLIB Instances}
    \label{t3}
    \setlength\tabcolsep{0.9pt}
    \begin{tabular*}{\textwidth}{@{\extracolsep\fill}llccccccccccl}
\toprule  
~& ~&\multicolumn{5}{c}{CPU Time(s)} & \multicolumn{5}{c}{GAP(\%)} &Fall \\ \cmidrule{3-7}\cmidrule{8-12}%
QPLIB &n & PPA & PABB & ALM &MPEC & SDR & PPA & PABB & ALM&MPEC & SDR  &  ALM \\
\midrule
     QPLIB5881 & 120 & 0.48 & \textbf{0.17} & 0.44 & 0.32 & 0.60 & 0.40 & \textbf{0.39} & 3.05 & 1.16 & 3.49 & 7 \\ 
        QPLIB5882 & 150 & 5.40 & 0.61 & 0.83 & \textbf{0.35} & 0.62 & \textbf{0.12} & \textbf{0.12} & 1.65 & 1.46 & 2.98 & 4 \\ 
        QPLIB5875 & 200 & 10.81 & 1.17 & 1.26 & \textbf{0.38} & 0.91 & \textbf{0.00} & 0.04 & 1.12 & 0.75 & 3.00 & 5 \\ 
        QPLIB3852 & 231 & 1.61 & 0.44 & 0.38 & \textbf{0.37} & 1.03 & \textbf{0.34} & \textbf{0.34} & 1.45 & 2.56 & 5.13 & 0 \\ 
        QPLIB5909 & 250 & 1.44 & 0.52 & 0.85 & \textbf{0.35} & 1.19 & \textbf{0.27} & 0.52 & 3.59 & 1.45 & 5.88 & 4 \\ 
        QPLIB3565 & 276 & 1.45 & 0.30 & \textbf{0.29} & 0.30 & 0.87 & 1.35 & 1.28 & \textbf{1.26} & 2.27 & 8.51 & 1 \\ 
        QPLIB5721 & 300 & 14.24 & 5.81 & 3.41 & \textbf{0.66} & 1.37 & 2.74 & 1.70 & 1.01 & 1.93 & \textbf{0.94} & 6 \\ 
        QPLIB3745 & 325 & 2.35 & 0.60 & 0.52 & \textbf{0.43} & 1.37 & 1.14 & \textbf{1.08} & 1.86 & 2.40 & 5.39 & 1 \\ 
        QPLIB5725 & 343 & 1.05 & \textbf{0.50} & 1.25 & 0.58 & 1.57 & 4.51 & 2.38 & 2.12 & 2.27 & 11.09 & 7 \\ 
        QPLIB3705 & 378 & 2.78 & 0.72 & 0.54 & \textbf{0.44} & 1.94 & \textbf{0.62} & 1.04 & 2.19 & 1.04 & 9.90 & 0 \\ 
        QPLIB5755 & 400 & 1.49 & 0.54 & 1.34 & \textbf{0.47} & 2.09 & 3.78 & 2.19 & \textbf{1.72} & 2.17 & 11.08 & 6 \\
        QPLIB3738 & 435 & 3.88 & 0.63 & 1.26 & \textbf{0.49} & 2.30 & \textbf{1.52} & 1.66 & 6.30 & 2.84 & 4.74 & 3 \\ 
        QPLIB3506 & 496 & 3.55 & 0.79 & 0.59 & \textbf{0.43} & 2.29 & \textbf{0.84} & 0.92 & 1.57 & 1.67 & 10.88 & 2 \\ 
        QPLIB5922 & 500 & 9.71 & 1.64 & 2.07 & \textbf{0.41} & 3.91 & \textbf{0.15} & 0.18 & 1.38 & 0.56 & 1.92 & 5 \\ 
        QPLIB3832 & 561 & 4.48 & 0.91 & 0.58 & \textbf{0.49} & 3.61 & \textbf{1.12} & 2.02 & 1.77 & 2.53 & 6.86 & 0 \\ 
        QPLIB3877 & 630 & 5.15 & 0.97 & 0.71 & \textbf{0.55} & 4.53 & \textbf{1.10} & 1.40 & 2.18 & 1.33 & 7.31& 1 \\ 
        QPLIB3706 & 703 & 6.17 & 0.95 & 0.78 & \textbf{0.54} & 5.17 & 2.11 & \textbf{1.50} & 1.97 & 2.05 & 5.87 & 3 \\ 
        QPLIB3838 & 780 & 7.47 & 1.07 & 0.95 & \textbf{0.59} & 7.41 & \textbf{2.36} & 2.79 & 3.25 & 2.95 & 8.58 & 1 \\ 
        QPLIB3822 & 861 & 7.64 & 1.20 & 1.15 & \textbf{0.56} & 8.37 & 1.34 & \textbf{0.85} & 2.09 & 1.18 & 3.76 & 0 \\ 
        QPLIB3650 & 946 & 8.22 & 1.75 & 1.23 & \textbf{0.64} & 10.86 & 1.02 & 1.46 & 2.28 & \textbf{0.65} & 5.00 & 4 \\ 
        QPLIB3642 & 1035 & 7.05 & 1.48 & 1.11 & \textbf{0.56} & 10.27 & 1.59 & 1.80 & 2.36 & \textbf{1.16} & 8.70 & 1 \\ 
        QPLIB3693 & 1128 & 10.14 & 2.02 & 1.63 & \textbf{0.67} & 14.90 & \textbf{0.80} & 1.39 & 1.44 & 1.92 & 8.01 & 2 \\ 
        QPLIB3850 & 1225 & 10.41 & 1.61 & 1.91 & \textbf{0.75} & 18.35 & \textbf{0.96} & 1.50 & 1.29 & 1.53 & 6.72 & 1 \\ \hline
        Average & ~ & 5.52 & 1.15 & 1.09 & \textbf{0.49} & 4.59 & 1.31 & \textbf{1.24} & 2.13 & 1.73 & 6.34 & 3  \\ 
 \hline
\end{tabular*}
\end{table}

\textbf{(GKA)} These datasets can also be obtained from the OR-Library\footnote{http://people.brunel.ac.uk/$\sim$mastjjb/jeb/info.html}, denoted as sets a, b, c, d, e, f. The dataset parameters are generated using the Pardalos-Rodgers generators introduced by Pardalos and Rodgers~\cite{Pardalos1990ComputationalAO}. The dimensions of this dataset range from 40 to 500, with densities ranging from 0 to 1.

It is clear from Table \ref{t5} that for \( n \leq 100 \), PABB takes the least amount of computational time. However, as \( n \) increases to 200 and 500, the computational time of PABB exceeds that of MPEC. The solution quality produced by PPA and PABB is generally superior to that of the other methods, particularly in many instances where \( n \leq 200 \), their GAPs are one-tenth that of MPEC.

\begin{table}[htbp]
    \centering
    \caption{Experimental results of GKA Instances}
    \label{t5}
    \setlength\tabcolsep{0.88pt}
    \begin{tabular*}{\textwidth}{@{\extracolsep\fill}llccccccccccl}
\toprule  
~& ~&\multicolumn{5}{c}{CPU Time(s)} & \multicolumn{5}{c}{GAP(\%)} &Fall \\ \cmidrule{3-7}\cmidrule{8-12}%
GKA&n & PPA & PABB & ALM &MPEC & SDR & PPA & PABB & ALM&MPEC & SDR  &  ALM \\
\midrule
  gka1a & 50 & 0.07 & \textbf{0.03} & 0.15 & 0.19 & 0.29 & 0.527 & 0.539 & 0.849 & \textbf{0.000} & 0.996 & 9 \\ 
        gka2a & 60 & 0.04 & \textbf{0.02} & 0.19 & 0.22 & 0.36 & \textbf{0.000} & \textbf{0.000} & 0.099 & \textbf{0.000} & 1.039 & 4 \\ 
        gka3a & 70 & 0.11 & \textbf{0.05} & 0.19 & 0.25 & 0.30 & \textbf{0.023} & 0.045& 0.182 & 0.646 & 2.418 & 3 \\ 
        gka4a & 80 & 0.04 & \textbf{0.02} & 0.22 & 0.25 & 0.32 & \textbf{0.000} & 0.026 & 0.803 & 0.233 & 0.640 & 1 \\ 
        gka5a & 50 & 0.07 & \textbf{0.02} & 0.22 & 0.30 & 0.40 & 0.052 &\textbf{0.047} & 0.854 & 0.052 & 2.911 & 5 \\ 
        gka6a & 30 & 0.04 & \textbf{0.02} & 0.15 & 0.28 & 0.37 & \textbf{0.000} & 0.018 & 6.583 & \textbf{0.000} & 0.176 & 8 \\ 
        gka7a & 30 & 0.04 & \textbf{0.02} & 0.14 & 0.28 & 0.37 & \textbf{0.000} & \textbf{0.000} & 2.510 & \textbf{0.000} & \textbf{0.000} & 5 \\ 
        gka8a & 100 & 0.07 & \textbf{0.02} & 0.28 & 0.33 & 0.51 & \textbf{0.000} & \textbf{0.000} & 0.216 & \textbf{0.000} & 0.522 & 3 \\ 
        gka1c & 40 & 0.04 & \textbf{0.01} & 0.12 & 0.23 & 0.25 & \textbf{0.000} & \textbf{0.000} & \textbf{0.000} & \textbf{0.000} & \textbf{0.000} & 7 \\ 
        gka2c& 50 & 0.07 & \textbf{0.02} & 0.19 & 0.29 & 0.31 & \textbf{0.000} & \textbf{0.000} & \textbf{0.000} & 0.306 & 0.998 & 5 \\ 
        gka3c & 60 & 0.08 & \textbf{0.03} & 0.24 & 0.27 & 0.42 & \textbf{0.000} & 0.008 & 0.900 & 0.255 & 0.255 & 4 \\ 
        gka4c & 70 & 0.05 & \textbf{0.02} & 0.16 & 0.18 & 0.29 & \textbf{0.000} & \textbf{0.000} & 0.378 & 0.068 & 0.825 & 8 \\ 
        gka5c & 80 & 0.12 & \textbf{0.06} & 0.25 & 0.31 & 0.44 & \textbf{0.000} & \textbf{0.000} & 1.616 & 0.435 & 1.372 & 6 \\ 
        gka6c & 90 & 0.10 & \textbf{0.05} & 0.27 & 0.30 & 0.48 & \textbf{0.000} & \textbf{0.000} & 0.017& 0.343 & 1.185 & 3 \\ 
        gka7c & 100 & 0.12 & \textbf{0.05} & 0.26 & 0.44 & 0.53 & \textbf{0.000} & \textbf{0.000} & 0.720 & 1.522 & 1.107& 5 \\ 
        gka1d & 100 & 0.12 & \textbf{0.04} & 0.24 & 0.19 & 0.37 & 0.174 & \textbf{0.025} & 1.026 & 2.447 & 2.100 & 3 \\ 
        gka2d & 100 & 0.29 & \textbf{0.09} & 0.33 & 0.25 & 0.48 & \textbf{0.108} & 0.134 & 1.844 & 1.170 & 7.098 & 4 \\ 
        gka3d & 100 & 0.30 & \textbf{0.10} & 0.34 & 0.37 & 0.52 & 0.173 & \textbf{0.080} & 5.129 & 0.173 & 1.447 & 4 \\ 
        gka4d & 100 & 0.25 & \textbf{0.06} & 0.24 & 0.19 & 0.35 & \textbf{0.021} & 0.062 & 1.231 & 0.121 & 1.603 & 3 \\ 
        gka5d & 100 & 0.36 & \textbf{0.08} & 0.31 & 0.36 & 0.52 & \textbf{0.095} & 0.415 & 1.338 & 3.140 & 4.602 & 4 \\ 
        gka6d & 100 & 0.21 & \textbf{0.07} & 0.29 & 0.34 & 0.46 & 0.352 & \textbf{0.000} & 2.119& 1.330 & 3.076 & 3 \\ 
        gka7d & 100 & 0.40 & \textbf{0.09} & 0.28 & 0.35 & 0.51 & 0.015 & 0.089 & 2.344 & \textbf{0.000} & 4.255 & 4 \\ 
        gka8d & 100 & 0.25 & \textbf{0.04} & 0.30 & 0.34 & 0.49 & \textbf{0.000} & 0.217 & 1.137 & 0.183 & 2.312 & 3 \\ 
        gka9d & 100 & 0.33 & \textbf{0.07} & 0.30 & 0.49 & 0.48 & \textbf{0.006} & 0.011 & 1.234 & 1.207& 2.549 & 5 \\ 
        gka10d & 100 & 0.45 & \textbf{0.08} & 0.30 & 0.21 & 0.40 & \textbf{0.042} & 0.138 & 0.853 & 0.832 & 1.335 & 2 \\ 
        gka1e & 200 & 0.85 & \textbf{0.23} & 0.68 & 0.25 & 0.81 & \textbf{0.081} & 0.165 & 1.252 & 1.039 & 1.761 & 3 \\ 
        gka2e & 200 & 1.77 & 0.33 & 1.00 & \textbf{0.26} & 1.00 & \textbf{0.086} & 0.102 & 1.888& 0.816 & 2.013 & 5 \\ 
        gka3e & 200 & 1.99 & 0.33 & 0.95 & \textbf{0.28} & 1.00 & 0.136 & \textbf{0.057} & 2.733 & 0.382 & 2.892& 1 \\ 
        gka4e & 200 & 1.89 & 0.24 & 0.77 & \textbf{0.23} & 0.82 & 0.044 & \textbf{0.034} & 1.992 & 0.337 & 1.365 & 5 \\ 
        gka5e & 200 & 3.18 & \textbf{0.35} & 0.85 & 0.39 & 0.94 & \textbf{0.071} & 0.132 & 2.402 & 1.155 & 5.402 & 3 \\ 
        gka1f & 500 & 11.34 & 2.12 & 4.60 & \textbf{0.33} & 3.42 & 3.581 & 3.669 & 4.750 & 4.363 & 6.883 & 5 \\ 
        gka2f & 500 & 20.15 & 1.91 & 4.77 & \textbf{0.30} & 3.63 & \textbf{4.601} & 4.691 & 6.847 & 5.351\ & 6.832 & 4 \\ 
        gka3f & 500 & 33.20 & 2.47 & 5.18 & \textbf{0.34} & 2.81 & 5.238 & \textbf{5.212} & 7.272 & 6.245 & 8.394 & 4 \\ 
        gka4f & 500 & 37.94 & 2.37 & 5.07 & \textbf{0.45} & 3.58 & 4.945 & \textbf{4.942} & 8.309 & 6.032 & 7.902 & 5 \\ 
        gka5f & 500 & 43.53 & 2.10 & 5.13 & \textbf{0.51} & 3.65 & \textbf{5.321} & 5.450 & 9.637 & 5.724 & 9.563 & 4 \\ \hline
        AVERAGE & ~ & 4.57 & 0.39 & 1.00 & \textbf{0.30} & 0.91 & \textbf{0.734} & 0.752 & 2.316 & 1.312 & 2.795 & 4 \\ \hline
\end{tabular*}
\end{table}

\textbf{(BE)}
There are a total of 8 categories, each with 10 instances in Billionnet and Elloumi instances. These instances have dimensions of 20, 40, 60, and 80, with densities ranging from 0.1 to 1.0. The diagonal coefficients of the instances fall within the range [-100, 100], while the off-diagonal coefficients range from [-50, 50]. In Table~\ref{t4}, we consider the solution presented in the article~\cite{BENLIC20131162} as the lower bound (lb).

It is evident that both PABB and MPEC have shorter computation times compared to other methods. In the majority of cases, PABB has a shorter runtime than MPEC. Furthermore, PABB and PPA show a clear advantage in terms of GAP compared to other methods, with a GAP typically below 0.1\%.
\begin{table}[htbp]
 \centering
    \caption{Experimental results of BE Instances}
    \label{t4}
    \setlength\tabcolsep{0.9pt}
    \begin{tabular*}{\textwidth}{@{\extracolsep\fill}llccccccccccl}
\toprule  
~& ~&\multicolumn{5}{c}{CPU Time(s)} & \multicolumn{5}{c}{GAP(\%)} &Fall \\ \cmidrule{3-7}\cmidrule{8-12}%
n &Density& PPA & PABB & ALM &MPEC & SDR & PPA & PABB & ALM&MPEC & SDR  &  ALM \\
\midrule
   100 & 1.0 & 0.55 & \textbf{0.09} & 0.38 & 0.30 & 0.50 & \textbf{0.01} & 0.07 & 1.45 & 0.67 & 2.76 & 3.70 \\ 
        120 & 0.3 & 0.58 & \textbf{0.14} & 0.54 & 0.28 & 0.58 & \textbf{0.07} & 0.09 & 2.09 & 0.92 & 2.43 & 2.70 \\ 
        120 & 0.8 & 1.16 & \textbf{0.17} & 0.57 & 0.28 & 0.55 & \textbf{0.04} & 0.07 & 1.95 & 1.28 & 3.47 & 3.90 \\ 
        150 & 0.3 & 1.34 & \textbf{0.25} & 0.75 & 0.34 & 0.83 & \textbf{0.07} & 0.14 & 2.01 & 1.17 & 3.30& 2.80 \\ 
        150 & 0.8 & 2.23 & \textbf{0.28} & 0.72 & 0.35 & 0.72 & 0.10 & \textbf{0.10} & 2.91 & 0.86 & 3.24 & 4.50 \\ 
        200 & 0.3 & 2.64 & 0.39 & 1.05 & \textbf{0.32} & 1.17 & 0.10 & \textbf{0.09} & 1.90 & 0.99 & 3.51 & 3.80 \\ 
        200 & 0.8 & 4.18 & 0.40 & 1.05 & \textbf{0.37} & 0.91 & \textbf{0.05} & 0.11 & 2.22 & 0.96 & 3.27 & 3.60 \\ 
        250 & 0.1 & 2.00 & \textbf{0.44} & 1.16 & 0.29 & 1.20 & \textbf{0.05} & 0.10 & 1.23 & 1.02 & 2.78 & 4.00 \\ 
\hline
\end{tabular*}
\end{table}
\subsection{Max cut}
The max-cut problem can be formulated as
$$
\max \left\{x^{T} W x: x \in\{-1,1\}^{n}\right\},
$$
where  $W$ is the weighted adjacency matrix of the given graph G.  
We test max cut problem on rudy and G set datasets.

\textbf{(rudy)}
Graphs of the following four types have been generated using rudy~\cite{rinaldi1998rudy}:
\begin{itemize}
\item $G_{0.5}$, $g05_n.i$ unweighted graphs with edge probability  1/2, n=60,80,100;
    \item  $G_{-1 / 0 / 1}$,  $pm1s_n.i$, $pm1d_n.i$
weighted graph with edge weights chosen uniformly from  \{-1,0,1\}  and density  10\%  and  99\%  respectively, n=80,100;
\item  $G_{[-10,10]} ,\mathrm{w} d \_n \cdot i$  Graph with integer edge weights chosen from  [-10,10]  and density  d=0.1,0.5,0.9, n=100;
\item $G_{[0,10]}, \mathrm{pw} d \_n . i$  Graph with integer edge weights chosen from  [-10,10]  and density  d=0.1,0.5,0.9, n=100.
\end{itemize}

All five algorithms successfully solve all instances, with each instance being run once. From Tables~\ref{t6} we can see that PPA and PABB consistently have the smallest GAP, with PABB requiring less CPU time. In addition, ALM occasionally has the shortest runtime, but its solution quality is relatively inferior.

\begin{table}[htbp]
    \centering
    \caption{Experimental results of $rudy_{g05},rudy_{PM},rudy_{W},rudy_{PW}$}
    \label{t6}
    \setlength\tabcolsep{0.9pt}
    \begin{tabular*}{\textwidth}{@{\extracolsep\fill}llcccccccccc}
\toprule  
~& ~&\multicolumn{5}{c}{CPU Time(s)} & \multicolumn{5}{c}{GAP(\%)}  \\ \cmidrule{3-7}\cmidrule{8-12}%
n&Density & PPA & PABB & ALM &MPEC & SDR & PPA & PABB & ALM&MPEC & SDR    \\
\midrule
$\boldsymbol{rudy_{g05}}$\\
        60 & - & 0.78 & 0.15 & \textbf{0.12} & 0.41 & 0.49 & \textbf{0.06} & 0.11 & 0.19 & 0.47 & 2.58 \\ 
        80 & - & 1.55 & 0.26 & \textbf{0.14} & 0.43 & 0.51 & \textbf{0.10} & 0.16 & 0.64 & 0.18 & 2.24 \\ 
        100 & - & 2.64 & 0.45 & \textbf{0.21} & 0.52 & 0.56 & \textbf{0.03} & 0.05 & 0.41 & 0.37 & 2.42 \\  \hline
$\boldsymbol{rudy_{PM}}$\\        
        80 & 0.1 & 0.74 & 0.26 & \textbf{0.15} & 0.54 & 0.76 & 1.46 & \textbf{0.74} & 2.28 & 2.16 & 13.54 \\ 
        80 & 0.99 & 2.00 & 0.32 & \textbf{0.25} & 0.47 & 0.57 & \textbf{0.17} & 0.58 & 2.72 & 1.95 & 16.93 \\ 
        100 & 0.1 & 0.93 & 0.25 & \textbf{0.23} & 0.57 & 0.94 & \textbf{0.72} & 1.03 & 3.58 & 2.25 & 16.16 \\ 
        100 & 0.99 & 3.56 & 0.49 & \textbf{0.27} & 0.47 & 0.56 & \textbf{0.13} & 1.08 & 5.18 & 1.74 & 17.79 \\ 
        \hline
    $\boldsymbol{rudy_{W}}$\\ 
        100 & 0.1 & 0.32 & \textbf{0.17} & 0.26 & 0.47 & 0.64 & \textbf{0.57} & 0.74 & 1.65 & 1.95 & 14.32 \\ 
        100 & 0.5 & 1.49 & \textbf{0.27} & 0.31 & 0.44 & 0.55 & \textbf{0.44} & 1.20 & 3.78 & 2.88 & 17.10 \\ 
        100 & 0.9 & 2.84 & \textbf{0.43} & 0.44 & 0.49 & 0.56 & \textbf{1.05} & 1.33 & 6.37 & 2.52 & 16.25 \\ 
                \hline
    $\boldsymbol{rudy_{PW}}$\\ 
        100 & 0.1 & 0.39 & \textbf{0.16} & 0.26 & 0.46 & 0.58 & \textbf{0.29} & 0.62 & 0.56 & 1.48 & 4.23 \\ 
        100 & 0.5 & 1.87 & 0.41 & \textbf{0.35} & 0.50 & 0.52 & 0.12 &\textbf{0.06} & 1.04 & 0.38 & 2.64 \\ 
        100 & 0.9 & 4.45 & 0.62 & \textbf{0.47} & 0.62 & 0.50 & 0.06 & \textbf{0.02} & 4.44 & 0.24 & 1.95 \\ 
                \hline
\end{tabular*}
\end{table}
\textbf{(G sets)}
The instances in the G set are numbered in ascending order by dimension. We selected the first five instances for the 800, 1000, and 2000 dimensions. For the 3000-dimensional instances, there are only 3 total instances in the G set, so we included all of them in the test. The optimal solution (lb) is taken from the article~\cite{BENLIC20131162}.

As shown in Table~\ref{t10}, PABB has a shorter runtime compared to PPA and SRD, similar to ALM, but longer than MPEC. Overall, the solution quality of PABB is close to that of PPA, outperforming ALM, MPEC, and SDR, especially for the 1000-dimensional instances. In the case of the 3000-dimensional instances, ALM and SDR fail, while PPA and PABB prove to be more effective in solving the problems.

\begin{table}[htbp]
    \centering
    \caption{Experimental results of G sets}
    \label{t10}
    \setlength\tabcolsep{0.9pt}
    \begin{tabular*}{\textwidth}{@{\extracolsep\fill}llcccccccccc}
\toprule  
~& ~&\multicolumn{5}{c}{CPU Time(s)} & \multicolumn{5}{c}{GAP(\%)}  \\ \cmidrule{3-7}\cmidrule{8-12}%
Gset&n & PPA & PABB & ALM &MPEC & SDR & PPA & PABB & ALM&MPEC & SDR    \\
\midrule
G16 & 800 & 4.50 & 1.79 & 1.57 & \textbf{0.38} & 6.89 & 1.70 & 1.54 & \textbf{1.28} & 2.26 & 3.93 \\ 
        G17 & 800 & 4.27 & 1.65 & 1.93 & \textbf{0.44} & 6.72 & \textbf{1.12} & 1.54 & 1.31 & 2.03 & 4.73 \\ 
        G18 & 800 & 6.78 & 1.33 & 0.99 & \textbf{0.40} & 6.37 & \textbf{1.51} & 2.62 & 3.33 & 6.96 & 14.42 \\ 
        G19 & 800 & 6.37 & 1.16 & 1.21 & \textbf{0.31} & 6.03 & \textbf{1.10} & 2.54 & 3.86 & 6.51 & 12.25 \\ 
        G20 & 800 & 8.77 & 1.52 & 1.47 & \textbf{0.36} & 6.86 & \textbf{1.91} & \textbf{1.91} & 3.72& 4.57 & 16.15 \\ 
        G43 & 1000 & 35.78 & 8.18 & 1.21 & \textbf{0.83} & 15.59 & 0.42 & \textbf{0.32} & 1.22 & 1.47 & 3.50 \\ 
        G44 & 1000 & 35.42 & 8.55 & 1.51 & \textbf{0.77} & 11.90 & \textbf{0.03} & 0.27 & 1.62 & 1.02 & 4.30 \\ 
        G45 & 1000 & 31.90 & 7.39 & 1.94 & \textbf{0.63} & 12.06 & \textbf{0.27} & 0.30 & 1.55 & 0.90 & 5.03 \\ 
        G46 & 1000 & 24.63 & 9.72 & 1.59 & \textbf{0.77} & 12.51 & \textbf{0.18} & 0.32 & 1.04 & 0.86 & 3.99 \\ 
        G47 & 1000 & 34.50 & 7.64 & 1.47 & \textbf{0.59} & 11.86 & 0.20 & \textbf{0.15} & 1.19 & 0.69 & 3.59 \\ 
        G36 & 2000 & 19.42 & 7.01 & 4.29 & \textbf{0.59} & 62.05 & 1.65 & 1.62 & \textbf{1.16} & 3.10 & 4.10 \\ 
        G37 & 2000 & 18.20 & 7.64 & 5.05 & \textbf{0.66} & 59.30 & 1.87 & 1.38 & \textbf{1.11} & 3.10 & 4.03 \\ 
        G38 & 2000 & 21.84 & 7.98 & 4.38 & \textbf{0.67} & 59.96 & 1.50 & \textbf{1.35} & 1.76 & 2.94 & 4.27 \\ 
        G39 & 2000 & 35.33 & 11.41 & 6.06 & \textbf{1.32} & 61.18 & \textbf{2.28} & 2.70 & 5.90 & 8.06 & 14.87 \\ 
        G40 & 2000 & 27.42 & 8.68 & 4.87 & \textbf{0.51} & 61.13 & \textbf{1.08} & 3.67 & 3.83 & 10.13 & 15.71 \\ 
        G48 & 3000 & 0.99 & \textbf{0.61} & - & 0.95 & - & 9.40 & \textbf{0.00} & \textbf{Fall} & \textbf{0.00} & \textbf{Fall} \\ 
        G49 & 3000 & 8.90 & \textbf{0.31} & - & 2.20 & - & \textbf{0.00} & \textbf{0.00} & \textbf{Fall} & 1.03 & \textbf{Fall} \\ 
        G50 & 3000 & \textbf{0.59} & 0.61 & - & 0.86 & - & 3.95 & 1.43 & \textbf{Fall} & \textbf{0.07} & \textbf{Fall} \\ 
\hline
\end{tabular*}
\end{table}

\section{Conclusion}

This paper presents a vectorized PSDP method for solving the UBQP problem, where the method framework is based on PSDP and involves vectorizing matrix variables to significantly reduce computational complexity. Various enhancements have been made in algorithmic details, such as penalty factor updates and initialization. Two algorithms are proposed for solving subproblems, each incorporating proximal point method and projection alternating BB method. Properties of the penalty function and some aspects of algorithm convergence are discussed. Numerical experiments demonstrate that the proposed method not only achieves satisfactory solution times compared to SDR and other established methods but also demonstrates outstanding capability in obtaining high-quality solutions.

In future research, our goal is to delve deeper into strategies for accelerating the solution of the subproblem, with the aim of increasing the solution speed. In addition, we aim to synchronize the adjustment of the penalty parameter with the updates of the subproblem, thereby reducing the number of iterations required to solve the subproblem.

\begin{acknowledgements}
This work was funded by National Key R\&D Program of China grant No. 2022YFA1003800, NSFC grant No. 12201318 and the Fundamental Research Funds for the Central Universities No. 63243075.
\end{acknowledgements}

\bibliographystyle{spmpsci}
\bibliography{references}

\end{document}